\documentclass[11pt, a4paper,leqno]{amsart}
\usepackage{amsmath,amsthm,amscd,amssymb,amsfonts, amsbsy}
\usepackage{latexsym}
\usepackage{exscale}
\usepackage{mathtools}

\usepackage{pgf}

\calclayout
\allowdisplaybreaks

\newtheorem{theorem}{Theorem}[section]

\newtheorem{lemma}{Lemma}[section]

\theoremstyle{definition}
\newtheorem{definition}{Definition}

\numberwithin{equation}{section}
\def\Xint#1{\mathchoice
{\XXint\displaystyle\textstyle{#1}}%
{\XXint\textstyle\scriptstyle{#1}}%
{\XXint\scriptstyle\scriptscriptstyle{#1}}%
{\XXint\scriptscriptstyle%
\scriptscriptstyle{#1}}%
\!\int}
\def\XXint#1#2#3{{\setbox0=\hbox{$#1{#2#3}{%
\int}$ }
\vcenter{\hbox{$#2#3$ }}\kern-.6\wd0}}
\def\barint{\,\Xint -} 
\def\bariint{\barint_{} \kern-.4em \barint}
\def\bariiint{\bariint_{} \kern-.4em \barint}

\newcommand{\beq}{\begin{equation}}
\newcommand{\bea}[1]{\begin{array}{#1} }
\newcommand{\eeq}{ \end{equation}}
\newcommand{\ea}{ \end{array}}

\newcommand{\mcl}{\mathcal}

\def \d {{\delta}}

\def\mean#1{\mathchoice%
          {\mathop{\kern 0.2em\vrule width 0.6em height 0.69678ex depth -0.58065ex
                  \kern -0.8em \intop}\nolimits_{\kern -0.4em#1}}%
          {\mathop{\kern 0.1em\vrule width 0.5em height 0.69678ex depth -0.60387ex
                  \kern -0.6em \intop}\nolimits_{#1}}%
          {\mathop{\kern 0.1em\vrule width 0.5em height 0.69678ex
              depth -0.60387ex
                  \kern -0.6em \intop}\nolimits_{#1}}%
          {\mathop{\kern 0.1em\vrule width 0.5em height 0.69678ex depth -0.60387ex
                  \kern -0.6em \intop}\nolimits_{#1}}}

\def\vintslides_#1{\mathchoice%
          {\mathop{\kern 0.1em\vrule width 0.5em height 0.697ex depth -0.581ex
                  \kern -0.6em \intop}\nolimits_{\kern -0.4em#1}}%
          {\mathop{\kern 0.1em\vrule width 0.3em height 0.697ex depth -0.604ex
                  \kern -0.4em \intop}\nolimits_{#1}}%
          {\mathop{\kern 0.1em\vrule width 0.3em height 0.697ex depth -0.604ex
                  \kern -0.4em \intop}\nolimits_{#1}}%
          {\mathop{\kern 0.1em\vrule width 0.3em height 0.697ex depth -0.604ex
                  \kern -0.4em \intop}\nolimits_{#1}}}

\newcommand{\aveint}[2]{\mathchoice%
          {\mathop{\kern 0.2em\vrule width 0.6em height 0.69678ex depth -0.58065ex
                  \kern -0.8em \intop}\nolimits_{\kern -0.45em#1}^{#2}}%
          {\mathop{\kern 0.1em\vrule width 0.5em height 0.69678ex depth -0.60387ex
                  \kern -0.6em \intop}\nolimits_{#1}^{#2}}%
          {\mathop{\kern 0.1em\vrule width 0.5em height 0.69678ex depth -0.60387ex
                  \kern -0.6em \intop}\nolimits_{#1}^{#2}}%
          {\mathop{\kern 0.1em\vrule width 0.5em height 0.69678ex depth -0.60387ex
                  \kern -0.6em \intop}\nolimits_{#1}^{#2}}}

\def\eqn#1$$#2$${\begin{equation}\label#1#2\end{equation}}
\def\charfn_#1{{\raise1.2pt\hbox{$\chi
_{\kern-1pt\lower3pt\hbox{{$\scriptstyle#1$}}}$}}}

\def\qq1{q_*}
\def\q2{q_{**}}
\def\dist{\operatorname{dist}}

\newdimen\vintbar
\def\vint{-\kern-\vintbar\int}

\def\A{\mathcal A}

\def\J{\mathcal J}
\def\L{\mathcal L}
\def\K{\mathcal K}

\def\L{\mathcal L}

\def\0{\boldsymbol 0}

\newcommand{\Ll}{\left}
\newcommand{\Rr}{\right}

\newcommand{\R}{\mathbb R}

\newcommand{\eps}{\epsilon}

\newtoks\by
\newtoks\paper
\newtoks\book
\newtoks\jour
\newtoks\yr
\newtoks\pages
\newtoks\vol
\newtoks\publ

\def\name[#1, #2]{#1 #2}
\def\ota{{\hbox{\bf ???}}}
\def\cLear{\by=\ota\paper=\ota\book=\ota\jour=\ota\yr=\ota
\pages=\ota\vol=\ota\publ=\ota}
\def\endpaper{\the\by, \textit{\the\paper},
{\the\jour} \textbf{\the\vol} (\the\yr), \the\pages.\cLear}
\def\endbook{\the\by, \textit{\the\book},
\the\publ, \the\yr.\cLear}
\def\endpap{\the\by, \textit{\the\paper}, \the\jour.\cLear}
\def\endproc{\the\by, \textit{\the\paper}, \the\book, \the\publ,
\the\yr, \the\pages.\cLear}

\renewcommand{\d}{\, \mathrm{d}} 

\DeclareMathOperator*{\argmin}{arg\min}

\newcommand{\g}{\mathbf{g}}
\renewcommand{\j}{\mathbf{j}}

\newcommand{\la}{\langle}
\newcommand{\ra}{\rangle}

\begin{document}

\title[weak solutions for  Kolmogorov-Fokker-Planck]{The Dirichlet problem for  Kolmogorov-Fokker-Planck type equations with rough coefficients}

\address{Malte Litsg{\aa}rd \\Department of Mathematics, Uppsala University\\
S-751 06 Uppsala, Sweden}
\email{malte.litsgard@math.uu.se}

\address{Kaj Nystr\"{o}m\\Department of Mathematics, Uppsala University\\
S-751 06 Uppsala, Sweden}
\email{kaj.nystrom@math.uu.se}

\thanks{K. N was partially supported by grant  2017-03805 from the Swedish research council (VR)}

\author{Malte Litsg{\aa}rd and Kaj Nystr{\"o}m}
\maketitle
\begin{abstract}
\noindent \medskip
 We establish the existence and uniqueness, in bounded as well as unbounded Lipschitz type cylinders of the forms $U_X\times V_{Y,t}$ and $\Omega\times \mathbb R^{m}\times \mathbb R$, of weak solutions to Cauchy-Dirichlet problems for the strongly degenerate parabolic operator \begin{equation*}
\mathcal{L}:= \nabla_X\cdot(A(X,Y,t)\nabla_X)+X\cdot\nabla_Y-\partial_t,
\end{equation*}
assuming that $A=A(X,Y,t)=\{a_{i,j}(X,Y,t)\}$ is a real $m\times m$-matrix valued, measurable function such that $A(X,Y,t)$ is symmetric, bounded and uniformly elliptic. Subsequently we solve the continuous Dirichlet problem and establish the representation of the solution using associated parabolic measures.  The paper is motivated, through our recent studies \cite{LN1,LN2,LN3}, by a growing need and interest  to gain a deeper understanding of the Dirichlet problem for the operator $\L$ in Lipschitz type domains. The key idea underlying our results is to prove, along the lines of Brezis and Ekeland \cite{BE1,BE2}, and inspired by the recent work of S. Armstrong and J-C. Mourrat \cite{AM} concerning variational methods for the kinetic Fokker-Planck equation, that the solution can be obtained as the minimizer of a uniformly convex functional.  \\

\noindent
2000  {\em Mathematics Subject Classification.}  35K65, 35K70, 35H20, 35R03.
\noindent

\medskip

\noindent
{\it Keywords and phrases: Kolmogorov equation, parabolic, ultraparabolic, hypoelliptic, operators in divergence form,  Dirichlet problem, Lipschitz domain, weak solutions, convex functional, Brezis and Ekeland, kinetic Fokker-Planck equation.}
\end{abstract}

    \setcounter{equation}{0} \setcounter{theorem}{0}

    \section{Introduction}

    This paper is devoted to a study of existence and uniqueness of weak solutions to Cauchy-Dirichlet problems for the strongly degenerate parabolic operator
\begin{equation}\label{e-kolm-nd}
    \L:= \nabla_X\cdot(A(X,Y,t)\nabla_X)+X\cdot\nabla_Y-\partial_t,
\end{equation}
where $(X,Y,t):=(x_1,...,x_{m},y_1,...,y_{m},t)\in \mathbb R^{m}\times\mathbb R^{m}\times\mathbb R=:\mathbb R^{N+1}$, $N:=2m$, $m\geq 1$. $A=A(X,Y,t)=\{a_{i,j}(X,Y,t)\}$ is assumed to be a real $m\times m$-matrix valued, measurable function such that $A(X,Y,t)$ is symmetric and
    \begin{equation}\label{eq:Aellip}
  \kappa^{-1} |\xi|^2 \leq \sum_{i,j=1}^{m} a_{i,j}(X,Y,t)\xi_i \xi_j \leq \kappa |\xi|^2,
\end{equation}
for some $1 \leq \kappa < \infty$ and for all $\xi \in \R^{m}$, $(X,Y,t)\in\mathbb R^{N+1}$.

We recall that the prototype for the operators in \eqref{e-kolm-nd}, i.e. the operator
    $$\K:=\nabla_X\cdot \nabla_X+X\cdot\nabla_Y-\partial_t,$$
 was originally introduced and studied by
Kolmogorov in a famous note published in 1934 in Annals of Mathematics, see \cite{K}. Kolmogorov noted that $\K$ is an example of a degenerate parabolic
operator having strong regularity properties and he proved that
\[
\K u = f \in C^\infty \implies u \in C^\infty.
\]
Using the terminology introduced by H{\"o}rmander, see \cite{Hm}, we say that $\K$ is hypoelliptic. Naturally, for operators as in \eqref{e-kolm-nd}, assuming only measurable coefficients and \eqref{eq:Aellip}, the methods of Kolmogorov and H{\"o}rmander can not be directly applied in a study of weak solutions and estimates thereof.
{Regularity theory for weak solutions, assuming only measurable coefficients and \eqref{eq:Aellip}, was developed in e.g. \cite{WangZhang} and \cite{GIMV}, where local H{\"o}lder continuity of weak solutions was proved.  The Dirichlet problem for degenerate parabolic equations as in \eqref{e-kolm-nd} have been studied under stronger assumptions on the coefficients, in particular see \cite{Manfredini} where the coefficients are assumed to be H{\"o}lder continuous. }

Let $U_X\subset\mathbb R^m$ be a bounded Lipschitz domain and let $V_{Y,t}\subset \mathbb R^{m}\times \mathbb R$ be a bounded domain with boundary which is $C^{1,1}$-smooth, {i.e. $C^{1,1}$ with respect to $Y$ as well as $t$.} Let $N_{Y,t}$ denote the outer unit normal to $V_{Y,t}$. We introduce
    \begin{eqnarray}\label{Kolb}
\quad\partial_{\K}(U_X\times V_{Y,t}):=(\partial U_X\times V_{Y,t})\cup\{(X,Y,t)\in \overline{U_X}\times \partial V_{Y,t}\mid (X,-1)\cdot N_{Y,t}>0\}.
    \end{eqnarray}
    $\partial_\K(U_X\times V_{Y,t})$ will be referred to as the Kolmogorov boundary of $U_X\times V_{Y,t}$, and the Kolmogorov boundary serves, in the context of the operator $\L$, as the natural substitute for the parabolic boundary used in the context of the Cauchy-Dirichlet problem for uniformly elliptic parabolic equations.

In this paper we first establish the existence and uniqueness of appropriate weak solutions to the Cauchy-Dirichlet problem
\begin{equation}\label{dpweak}
\begin{cases}
	\mathcal{L} u =g^*  &\text{in} \ U_X\times V_{Y,t}, \\
u=g & \text{on} \ \partial_\K(U_X\times V_{Y,t}).
\end{cases}
\end{equation}
We refer to Section \ref{Sec1} for the precise definition of the function spaces, i.e. $W(U_X\times V_{Y,t})$, $L_{Y,t}^2(V_{Y,t},{H}_X^{-1}(U_X))$, used, as well as the notions of weak solutions to \eqref{dpweak}, see Definition \ref{defwe}.

We prove the following theorem.

\begin{theorem}\label{weakdp1} Assume that $A$ satisfies \eqref{eq:Aellip} with constant $\kappa$. Consider $U_X\times V_{Y,t}\subset\mathbb R^{N+1}$,  where $U_X\subset\mathbb R^{m}$ is a bounded Lipschitz domain and
$V_{Y,t}\subset\mathbb R^{m}\times \mathbb R$ is a bounded domain with boundary  which is $C^{1,1}$-smooth. Let $g^*\in L_{Y,t}^2(V_{Y,t},{H}_X^{-1}(U_X))$, $g\in W(U_X\times V_{Y,t})$. Then there exists a unique weak solution $u\in W(U_X\times V_{Y,t})$ to the problem in \eqref{dpweak} in the sense of Definition \ref{defwe}. Furthermore, there exists a constant $c$, independent of $u$ and $g$, but depending on $m$, $\kappa$, and $U_X\times V_{Y,t}$, such that
    \begin{eqnarray*}
    ||u||_{W(U_X\times V_{Y,t})}\leq c\bigl (||g||_{W(U_X\times V_{Y,t})}+||g^*||_{L_{Y,t}^2(V_{Y,t},{H}_X^{-1}(U_X))}\bigr ).
    \end{eqnarray*}
    \end{theorem}

    Taking the analysis on $U_X\times V_{Y,t}$, and in particular Theorem \ref{weakdp1}, as the starting point we then consider weak solutions in the  unbounded setting $\Omega\times \mathbb R^m\times\mathbb R$ and in this case we assume that
$\Omega\subset\mathbb R^m$ is an (unbounded) Lipschitz domain
         \begin{eqnarray}\label{Lip}
 \Omega=\{ X=(x,x_{m})\in\mathbb R^{m-1}\times\mathbb R\mid x_m>\psi(x)\},
    \end{eqnarray}
    where $\psi:\mathbb R^{m-1}\to\mathbb R$ is a Lipschitz function. I.e. we consider the problem
     \begin{equation}\label{dpweak+gal}
\begin{cases}
	\L u = g^*  &\text{in} \ \Omega\times \mathbb R^{m}\times \mathbb R, \\
      u = g  & \text{on} \ \partial\Omega\times \mathbb R^{m}\times \mathbb R.
\end{cases}
\end{equation}
    In this case we refer to Section \ref{Sec3} for the precise definition of the function spaces, i.e.
$W(\mathbb R^{N+1})$, $W_{\mathrm{loc}}(\Omega\times \mathbb R^{m}\times \mathbb R)$, $ L^2_{Y,t}(\mathbb R^{m}\times \mathbb R,{H}_X^{-1}(\mathbb R^{m}))$, used as well as the notions of weak solutions, see Definition \ref{deffa}.

    In this setting we prove the following theorems: an existence theorem and a uniqueness theorem.

\begin{theorem}\label{weakdp1++} Assume that $A$ satisfies \eqref{eq:Aellip} with constant $\kappa$. Let $\Omega\subset\mathbb R^m$, $m\geq 1$, be a (unbounded) Lipschitz domain as defined in \eqref{Lip} and let $g\in W(\mathbb R^{N+1})$, $g^*\in L^2_{Y,t}(\mathbb R^{m}\times \mathbb R,{H}_X^{-1}(\mathbb R^{m}))$. Then there exists a weak solution $u\in W_{\mathrm{loc}}(\Omega\times \mathbb R^{m}\times \mathbb R)$ to the problem in \eqref{dpweak+gal} in the sense of Definition \ref{deffa}.
    \end{theorem}

\begin{theorem}\label{weakdp1gagain} Assume that $A$ satisfies \eqref{eq:Aellip} with constant $\kappa$.  Assume in addition that \begin{eqnarray}\label{comp}
    \quad A=A(X,Y,t)\equiv I_m\mbox{ outside some arbitrary but fixed compact subset of $\mathbb R^{N+1}$}.
    \end{eqnarray}
    Let $\Omega\subset\mathbb R^m$, $m\geq 1$, be a (unbounded) Lipschitz domain as defined in \eqref{Lip} and let $g\in W(\mathbb R^{N+1})\cap L^\infty(\mathbb R^{N+1})$, $g^*\in L^2_{Y,t}(\mathbb R^{m}\times \mathbb R,{H}_X^{-1}(\mathbb R^{m}))$. Let $u\in W_{\mathrm{loc}}(\Omega\times \mathbb R^{m}\times \mathbb R)$ be given by Theorem \ref {weakdp1++}.  If $u\in L^\infty(\Omega\times \mathbb R^{m}\times \mathbb R)$, then $u$ is the unique solution.
    \end{theorem}

    Finally, we consider the continuous Dirichlet problem and the representation of the solution using associated parabolic measures.

     \begin{theorem}\label{DP} Assume that $A$ satisfies \eqref{eq:Aellip} with constant $\kappa$ as well as \eqref{comp}. Let $\Omega\subset\mathbb R^m$, $m\geq 1$, be a (unbounded) Lipschitz domain as defined in \eqref{Lip} and let $\varphi\in W(\mathbb R^{N+1})\cap C_0(\mathbb R^{N+1})$. Let
$u=u_\varphi$ be the unique  bounded weak solution given by Theorem \ref{weakdp1++} and Theorem \ref{weakdp1gagain} with $g=\varphi$, $g^\ast\equiv 0$. Then
 $$u\in C(\bar \Omega\times \mathbb R^m\times\mathbb R).$$
Furthermore, there exists, for every $(X,Y,t)\in \Omega\times \mathbb R^m\times \mathbb R $, a unique probability
measure  $\omega_{\K}(X,Y,t,\cdot)$ on $\partial\Omega\times \mathbb R^m\times \mathbb R $ such that
\begin{eqnarray*}
u(X,Y,t)=\iiint_{\partial\Omega\times \mathbb R^m\times \mathbb R }\varphi(\tilde X,\tilde Y,\tilde t)\d \omega_{\K}(X,Y,t,\tilde X,\tilde Y,\tilde t).
\end{eqnarray*}
 \end{theorem}

This paper is motivated, through our recent  studies \cite{LN1,LN2,LN3}, by a growing need and interest  to gain a deeper understanding of the Dirichlet problem for the operator $\L$ in Lipschitz type domains. In \cite{LN1,LN2,LN3} we have developed a number of results and estimates concerning solutions to $\L u=0$ in Lipschitz type domains adapted to the dilation structure and the (non-Euclidean) Lie group underlying
    the operator $\L$. Indeed, in \cite{LN1} our results include scale and translation invariant boundary comparison principles,
boundary Harnack inequalities and doubling properties of associated parabolic measures. In \cite{LN2} we establish results concerning associated parabolic measures and their absolute continuity with respect to a surface measure, and in particular we study  when the associated Radon-Nikodym derivative defines an $A_\infty$-weight with respect to the surface measure. Finally, in \cite{LN3} we establish a structural theorem concerning the absolute continuity of elliptic and parabolic measures which allowed us to reprove several results previously established in the literature as well as deduce new results in, for example, the context of homogenization for operators of Kolmogorov type.  Our proof of the structural theorem is based on some of the results in \cite{LN1}  concerning boundary Harnack inequalities for operators of Kolmogorov type in divergence form with bounded, measurable and uniformly elliptic coefficients. An impetus for developing \cite{LN1,LN2,LN3} has been the recent results concerning  the local regularity of weak solutions to the equation in \eqref{dpweak} established in \cite{GIMV}. In \cite{GIMV} the authors extended the  De Giorgi-Nash-Moser (DGNM) theory, which in its original form only considers elliptic
or parabolic equations in divergence form, to hypoelliptic equations with rough coefficients including the operator $\L$. Their result is the correct scale- and translation-invariant estimates for local H{\"o}lder continuity and the Harnack inequality for weak solutions.

The point is that in \cite{LN1,LN2,LN3}, as well as \cite{GIMV}, questions concerning the existence and uniqueness of weak solutions to Dirichlet and Cauchy-Dirichlet problems for the operator $\L$ are constantly present and we have not been able to find, beyond what we have learned from \cite{AM}, coherent treatments in the literature of the problems in \eqref{dpweak}, \eqref{dpweak+gal}. This has been the main motivation for completing this work.

We prove Theorem \ref{weakdp1}  by proving that the solution to \eqref{dpweak} can be obtained as the minimizer of a uniformly convex functional. The fact that a parabolic equation can be cast as the first variation of a uniformly convex integral functional was first discovered in~\cite{BE1,BE2} and for a modern treatment of this approach, covering uniformly elliptic parabolic equations of second order in the more general context of uniformly monotone operators, we refer to \cite{ABM} which in turn is closely related to \cite{ghoussoub-tzou}, see also \cite{GBook}. More precisely, this paper is inspired by the interesting paper of S. Armstrong and J-C. Mourrat \cite{AM} concerning variational methods for the kinetic Fokker-Planck equation. In \cite{AM}, S. Armstrong and J-C. Mourrat consider {Kramers equation} and the time-dependent version of this equation which, in our coordinates, reads
\begin{equation}\label{e.thepde.witht}
\partial_t f -\Delta_X f + X \cdot \nabla_X f - X \cdot \nabla_Y f + {\bf b}\cdot\nabla_X f = f^* \quad \mbox{in}  \
\mathbb R^m\times U_Y\times\mathbb R_+.
\end{equation}
This equation is often referred to as the kinetic Fokker-Planck equation. In \cite{AM}, $f(v,x,t)$ is a function of velocity, position and time, in our notation $(v,x,t)=(X,Y,t)$. In \cite{AM}, S. Armstrong and J-C. Mourrat develop a functional analytic approach to the study of the Kramers and kinetic Fokker-Planck equations which parallels the classical $H^1$ theory of uniformly elliptic equations. In particular, they identify a function space analogous to $H^1$ and develop a well-posedness theory for weak solutions of the Dirichlet problem in this space. They prove new functional inequalities of Poincar\'e and H\"ormander type and use these to obtain the~$C^\infty$ regularity of weak solutions. They also use the Poincar\'e-type inequality established to give an elementary proof of the exponential convergence to equilibrium for solutions of the kinetic Fokker-Planck equation which mirrors the classic dissipative estimate for the heat equation.

The proof of Theorem \ref{weakdp1}  follows the proof developed in \cite{AM} for the Kramers and kinetic Fokker-Planck equations.  To compare our setting to that of \cite{AM},  we first note that the kinetic Fokker-Planck equation is different from the operator $\L$ and in \cite{AM} the authors consider domains of the form $\mathbb R^m\times V_{Y,t}$, where $V_{Y,t}$ is as above, or domains of the form $\mathbb R^m\times U_Y\times J$ where
$U_Y\subset\mathbb R^m$ is a bounded domain with $C^1$-boundary and $J\subset\mathbb R$ is a bounded interval. In this sense they consider unbounded domains and their natural function space in the variable $X$ is the Sobolev space which uses $\exp(-|X|^2/2)\d X$ as the underlying measure. The unbounded setting, and the Gaussian Poincar\'e inequality, force the authors in \cite{AM} to establish new Poincar\'e inequalities as a novel part of their argument. In this paper we initially work on truly bounded domains $U_X\times V_{Y,t}$ assuming that
$U_X\subset\mathbb R^m$ is a bounded Lipschitz domain and our result can be readily extended to domains of the form $U_X\times U_Y\times J$. Taking the analysis on $U_X\times V_{Y,t}$ as the starting point we then, as discussed, consider weak solutions in the  unbounded setting $\Omega\times \mathbb R^m\times\mathbb R$  assuming that $\Omega\subset\mathbb R^m$ is an (unbounded) Lipschitz domain in the sense of \eqref{Lip}. In particular, compared to \cite{AM} we in the unbounded case have a geometric restriction in the $X$-variable and none in the $(Y,t)$-variables. The latter geometric setting is, in particular, in line with the set up in \cite{LN3}.  In Theorem \ref{weakdp1gagain} and Theorem \ref{DP} the assumption in \eqref{comp} is only used in a qualitative fashion.

As a general comment, one reason to establish the existence and uniqueness of weak solutions in Sobolev spaces encoding integrability up to the boundary of weak derivatives, in this case for $\nabla_Xu$, is that one can then often relax the assumption on the coefficients to being only measurable, bounded and elliptic when proving boundary type estimates of the type considered in \cite{LN1,LN2,LN3}. As another general comment, we believe that it is a good project to develop versions of the estimates in \cite{LN1,LN2,LN3} in the geometric settings of $U_X\times V_{Y,t}$ and $U_X\times U_Y\times J$ and in particular near the subsets of
    $\partial_\K(U_X\times V_{Y,t})$ and $\partial_\K(U_X\times U_Y\times J)$ defined by
    \begin{eqnarray*}
\{(X,Y,t)\in \overline{U_X}\times \partial V_{Y,t}\mid (X,-1)\cdot N_{Y,t}>0\}\mbox{ and }\{(X,Y,t)\in \overline{U_X}\times \partial U_{Y}\times J\mid X\cdot N_{Y}>0\},
    \end{eqnarray*}
    where now $N_{Y}$ denotes the outer unit normal to $U_{Y}$.

    The rest of the paper is organized as follows. In Section \ref{Sec1} we, in the geometric setting of $U_X\times V_{Y,t}$, introduce function spaces, define weak solutions and
    prove Poincar\'e inequalities. While strictly speaking not necessary in our setting, we here prove a version on bounded domains, Lemma \ref{lemma1a} below, of the novel Poincar\'e inequality established in \cite{AM}. In Section \ref{Sec2} we prove Theorem \ref{weakdp1} following \cite{AM}.  In Section \ref{Sec3} we prove Theorems
    \ref{weakdp1++}-\ref{DP}.

    \section{Preliminaries}\label{Sec1}

    Throughout the section we let $U_X\subset\mathbb R^{m}$ and $V_{Y,t}\subset \mathbb R^{m}\times \mathbb R$  be bounded domains. {In addition we assume that $U_X\subset\mathbb R^{m}$ is a Lipschitz domain and that and $V_{Y,t}$ is a domain with boundary which is $C^{1,1}$-smooth.}

    \subsection{Function spaces} We denote by ${H}_X^1(U_X)$ the Sobolev
space of functions $g\in L_{}^2(U_X)$  whose distribution gradient in $U_X$ lies in $(L^2(U_X))^m$, i.e.
  \begin{eqnarray*}\label{fspace-}
{H}_X^1(U_X):=\{g\in L_{X}^2(U_X)\mid \nabla_Xg\in (L^2(U_X))^m\},
    \end{eqnarray*}
    and we set
    $$||g||_{{H}_X^1(U_X)}:=||g||_{L^2(U_X)}+|||\nabla_Xg|||_{L^2(U_X)},\ g\in {H}_X^1(U_X).$$
    We let ${H}_{X,0}^1(U_X)$ denote the closure of $C_0^\infty(U_X)$ in the norm of ${H}_X^1(U_X)$ and we recall, as $U_X$ is a bounded Lipschitz domain, that
    $C^\infty(\overline{U_X})$ is dense in ${H}_X^1(U_X)$. In particular, equivalently we could define ${H}_X^1(U_X)$ as the closure of $C^\infty(\overline{U_X})$
     in the norm $||\cdot||_{{H}_X^1(U_X)}$.
     {Note that as ${H}_{X,0}^1(U_X)$ is a Hilbert space it is reflexive, so that $({H}_{X,0}^1(U_X))^\ast=H_X^{-1}(U_X)$ and $(H_X^{-1}(U_X))^\ast={H}_{X,0}^1(U_X)$, where $()^\ast$ denotes the dual.
     Based on this we let ${H}_X^{-1}(U_X)$ denote the dual to ${H}_{X,0}^1(U_X)$ acting on functions in ${H}_{X,0}^1(U_X)$ through the
     duality pairing $\langle \cdot,\cdot\rangle:=\langle \cdot,\cdot\rangle_{H_X^{-1}(U_X),H_{X,0}^{1}(U_X)}$.}

     In analogy with the definition of ${H}_X^1(U_X)$, we let $W(U_X\times V_{Y,t})$ be the closure of $C^\infty(\overline{U_X\times V_{Y,t}})$ in the norm
     \begin{align}\label{weak1-+}
     ||u||_{W(U_X\times V_{Y,t})}&:=||u||_{L_{Y,t}^2(V_{Y,t},H_X^1(U_X))}+||(-X\cdot\nabla_Y+\partial_t)u||_{L_{Y,t}^2(V_{Y,t},{H}_X^{-1}(U_X))}\notag\\
     &:=\biggl (\iint_{V_{Y,t}}||u(\cdot,Y,t)||_{{H}_X^1(U_X)}^2\, \d Y\d t\biggl )^{1/2}\\
     &\quad +\biggl (\iint_{V_{Y,t}}||(-X\cdot\nabla_Y+\partial_t)u(\cdot,Y,t)||_{{H}_X^{-1}(U_X)}^2\, \d Y\d t\biggl )^{1/2}.\notag
    \end{align}
    In particular, $W(U_X\times V_{Y,t})$ is a Banach space and if $u\in W(U_X\times V_{Y,t})$ then \begin{eqnarray}\label{weak1-}
u\in L_{Y,t}^2(V_{Y,t},H_X^1(U_X))\mbox{ and } (-X\cdot\nabla_Y+\partial_t)u\in  L_{Y,t}^2(V_{Y,t},H_X^{-1}(U_X)).
    \end{eqnarray}
    Note that the dual of $L_{Y,t}^2(V_{Y,t},H_{X,0}^1(U_X))$, denoted by $(L_{Y,t}^2(V_{Y,t},H_{X,0}^1(U_X)))^\ast$, satisfies $$(L_{Y,t}^2(V_{Y,t},H_{X,0}^1(U_X)))^\ast=L_{Y,t}^2(V_{Y,t},H_X^{-1}(U_X)),$$
    and, as mentioned above,
    $$(L_{Y,t}^2(V_{Y,t},H_X^{-1}(U_X)))^\ast=L_{Y,t}^2(V_{Y,t},H_{X,0}^1(U_X)).$$

    The topological boundary of $U_X\times V_{Y,t}$ is denoted by $\partial(U_X\times V_{Y,t})$. Let $N_{Y,t}$ denote the outer unit normal to $V_{Y,t}$. We define a subset
    $\partial_{\K}(U_X\times V_{Y,t})\subset\partial(U_X\times V_{Y,t})$, the Kolmogorov boundary of $U_X\times V_{Y,t}$, as in \eqref{Kolb}. We let $C^\infty_{\K,0}(\overline{U_X\times V_{Y,t}})$, $C^\infty_{X,0}(\overline{U_X\times V_{Y,t}})$ and $C^\infty_{Y,t,0}(\overline{U_X\times V_{Y,t}})$ consist of the functions in
$C^\infty(\overline{U_X\times V_{Y,t}})$ which vanish on $\partial_{\K}(U_X\times V_{Y,t})$, $\{(X,Y,t)\in \partial{U_X}\times \overline{V_{Y,t}}\}$ and
$\{(X,Y,t)\in \overline{U_X}\times \partial V_{Y,t}\mid (X,-1)\cdot N_{Y,t}>0\}$, respectively.  We let $W_0(U_X\times V_{Y,t})$, $W_{X,0}(U_X\times V_{Y,t})$ and
$W_{Y,t,0}(U_X\times V_{Y,t})$ denote the closure in the norm of
$W(U_X\times V_{Y,t})$ of $C^\infty_{\K,0}(\overline{U_X\times V_{Y,t}})$, $C^\infty_{X,0}(\overline{U_X\times V_{Y,t}})$ and $C^\infty_{Y,t,0}(\overline{U_X\times V_{Y,t}})$, respectively.

We end the subsection with the following elementary lemma.
\begin{lemma}\label{funkid} Let $U_X$ and $V_{Y,t}$ be as above. Then
\begin{equation*}
    \begin{split}
        (i)\quad &W_{X,0}(U_X\times V_{Y,t})\subset
 L_{Y,t}^2(V_{Y,t},H_{X,0}^1(U_X)),\\
 (ii)\quad & L_{Y,t}^2(V_{Y,t},H_{X,0}^1(U_X))\cap W(U_X\times V_{Y,t}) \subset W_{X,0}(U_X\times V_{Y,t}).
    \end{split}
\end{equation*}

\end{lemma}
\begin{proof}
For any $\phi\in C^\infty_{X,0}(\overline{U_X\times V_{Y,t}})$ we have $\phi\in L_{Y,t}^2(V_{Y,t},H_{X,0}^1(U_X))$. By definition, given  $u\in W_{X,0}(U_X\times V_{Y,t})$ there exists a sequence $\phi_j\in C^\infty_{X,0}(\overline{U_X\times V_{Y,t}})$ such that
\[
\|u-\phi_j\|_{W(U_X\times V_{Y,t})}\rightarrow 0\mbox{ as }j\to\infty.
\]
But then
\begin{equation*}
    \| u-\phi_j \|_{L_{Y,t}^2(V_{Y,t},H_{X,0}^1(U_X))} \leq  \| u-\phi_j \|_{W(U_X\times V_{Y,t})}\rightarrow 0.
\end{equation*}
and hence $(i)$ follows.  To prove $(ii)$, simply note that
\begin{equation*}
L_{Y,t}^2(V_{Y,t},H_{X,0}^1(U_X))\cap W(U_X\times V_{Y,t})\cap C^\infty(\overline{U_X\times V_{Y,t}})
\end{equation*}
coincides with
\begin{equation*}
W_{X,0}(U_X\times V_{Y,t}) \cap C^\infty(\overline{U_X\times V_{Y,t}}).
\end{equation*}
This completes the proof.
\end{proof}

\subsection{Weak solutions}

\begin{definition}\label{defwe} Let $g\in W(U_X\times V_{Y,t})$ and $g^*\in L_{Y,t}^2(V_{Y,t},{H}_X^{-1}(U_X))$. We then say that $u$ is a weak solution to
    \begin{equation}\label{dpweak+}
\begin{cases}
	\L u = g^*  &\text{in} \ U_X\times V_{Y,t}, \\
      u = g  & \text{on} \ \partial_{\K}(U_X\times V_{Y,t}),
\end{cases}
\end{equation}
    if
                      \begin{eqnarray}\label{weak1}
u\in W(U_X\times V_{Y,t}),\ (u-g)\in  W_0(U_X\times V_{Y,t}),
    \end{eqnarray}
    and if
\begin{equation}\label{weak3}
    \begin{split}
     0 =&\iiint_{U_X\times V_{Y,t}}\ A(X,Y,t)\nabla_Xu\cdot \nabla_X\phi\, \d X \d Y \d t\\
    &+\iint_{V_{Y,t}}\ \langle g^*(\cdot,Y,t)+(-X\cdot\nabla_Y+\partial_t)u(\cdot,Y,t),\phi(\cdot,Y,t)\rangle\, \d Y \d t,
\end{split}
\end{equation}
for all $ \phi\in L_{Y,t}^2(V_{Y,t},H_{X,0}^1(U_X))$ and where $\langle \cdot,\cdot\rangle=\langle \cdot,\cdot\rangle_{H_X^{-1}(U_X),H_{X,0}^{1}(U_X)}$ is the duality pairing in $H_X^{-1}(U_X)$.
\end{definition}

Note that if $u$ is a weak solution to the equation $\L  u=0$  in $U_X\times V_{Y,t}$, then it is a weak solution in the sense of distributions, i.e.
                              \begin{eqnarray}\label{weak4}
                              \iiint_{}\ \bigl(A(X,Y,t)\nabla_Xu\cdot \nabla_X\phi-u(-X\cdot \nabla_Y+\partial_t)\phi\bigr )\, \d X \d Y \d t=0,
                               \end{eqnarray}
                               whenever $\phi\in C_0^\infty(U_X\times V_{Y,t})$.

\subsection{Poincar\'e inequalities}
\begin{lemma}\label{lemma1a-} There exists a constant $c$, $1\leq c<\infty$, depending only on $m$, and the diameter of $U_X$, such that
$$||f||_{L_{Y,t}^2(V_{Y,t},L^2(U_X))}\leq c||\nabla_Xf||_{L_{Y,t}^2(V_{Y,t},L^2(U_X))},$$
for all {$f\in L_{Y,t}^2(V_{Y,t},H_{X,0}^1(U_X))$}. 
\end{lemma}
\begin{proof} Consider {$f\in L_{Y,t}^2(V_{Y,t},H_{X,0}^1(U_X))$}. Then, using the standard Euclidean Poincar\'e inequality, we see that
\begin{equation*}  
\|f(\cdot,Y,t)\|_{L^2(U_X))} \leq c\|\nabla_X f(\cdot,Y,t)\|_{L^2(U_X)}.
\end{equation*}
Simply squaring and integrating with respect to $(Y,t)\in V_{Y,t}$ gives the result.
\end{proof}

\begin{lemma}\label{lemma1a} Assume that $0\in U_X$ and that the $m$-dimensional ball $B(0,2R)$ is contained in $U_X$. There exists a constant $c$, $1\leq c<\infty$, depending only on $m$, $U_X\times V_{Y,t}$ and $R$, such that
$$||f||_{L_{Y,t}^2(V_{Y,t},L^2(U_X))}\leq c\bigl (||\nabla_Xf||_{L_{Y,t}^2(V_{Y,t},L^2(U_X))}+||(-X\cdot\nabla_Y+\partial_t)f||_{L_{Y,t}^2(V_{Y,t},{H}_X^{-1}(U_X))}\bigr ),$$
for all {$f\in W_{Y,t,0}(U_X\times V_{Y,t})$.}
\end{lemma}
\begin{proof} Consider $f\in  W(U_X\times V_{Y,t})\cap C^\infty(\overline{U_X\times V_{Y,t}})$. In the following we let $\la f \ra_{U_X}$ and $\la f \ra_{U_X\times V_{Y,t}}$ denote the averages of $f=f(X,Y,t)$ over $U_X$ and $U_X\times V_{Y,t}$ with respect to $\d X$ and $\d X \d Y \d t$, respectively. To start the proof of the lemma we first note, using the standard Euclidean Poincar\'e inequality, that
\begin{equation*}  
\|f - \la f \ra_{U_X} \|_{L^2(V_{Y,t},L^2(U_X))} \leq c\|\nabla_X f\|_{L^2(V_{Y,t},L^2(U_X))}.
\end{equation*}

Next we are going to prove that there exists $c$, $1\leq c<\infty$, depending on  $m$, $U_X\times V_{Y,t}$ and $R$, such that if  $\phi \in C^\infty_0(V_{Y,t})$ satisfies
\begin{equation}
\label{e.bounds.phi.kin}
\|\phi\|_{L^2(V_{Y,t})} + \|\nabla_Y \phi\|_{L^2(V_{Y,t})} + \|\partial_t \phi\|_{L^2(V_{Y,t})} \le 1,
\end{equation}
then
\begin{align}
\label{e.bracket.H-1.kin}
&\Ll|\iint_{V_{Y,t}} \phi \, \partial_{t} \la f \ra_{U_X}\, \d Y\d t \Rr| + \sum_{i = 1}^m \Ll|\iint_{V_{Y,t}} \phi \, \partial_{y_i} \la f \ra_{U_X}\, \d Y\d t \Rr| \notag\\
&\le c \bigl( \|\nabla_X f\|_{L^2(V_{Y,t},L^2(U_X))} + \|(-X\cdot\nabla_Y+\partial_t)f  \|_{L^2(V_{Y,t},H_X^{-1}(U_X))} \bigr).
\end{align}
We start by proving the estimate for the time derivative of
$\la f \ra_{U_X}$. To do this  we select a smooth function $\xi_0 \in C^\infty_0(B(0,2R))$ such that
\begin{equation}
\label{e.prop.xi0}
\int_{B(0,2R)} \xi_0(X) \, \d X = 1 \quad \text{ and } \quad \int_{B(0,2R)} x_i \xi_0(X) \, \d X = 0\mbox{ for all $i\in\{1,...,m\}$}.
\end{equation}
Using these properties of $\xi_0$, we can write
\begin{align*}  
& \iint_{V_{Y,t}} \partial_t \phi(Y,t) \, \la f \ra_{U_X}(Y,t) \, \d Y\d t\\
&= \iiint_{U_X\times V_{Y,t}} \xi_0(X) \bigl(\partial_t \phi(Y,t) - X \cdot \nabla_Y \phi(Y,t)\bigr) \la f \ra _{U_X}(Y,t) \, \d X\d Y\d t\\
&=I_1+I_2,
\end{align*}
where
\begin{align*}  
I_1&:= \iiint_{U_X\times V_{Y,t}} \xi_0(X) \bigl(\partial_t - X \cdot \nabla_Y \bigr) \phi(Y,t) \,  f(X,Y,t) \, \d X\d Y\d t,\\
I_2&:= \iiint_{U_X\times V_{Y,t}} \xi_0(X) \bigl(\partial_t - X \cdot \nabla_Y \bigr) \phi(Y,t)  \bigl(\la f \ra_{U_X}(Y,t) - f(X,Y,t)\bigr) \, \d X\d Y\d t.
\end{align*}
By integration by parts, 
\begin{equation*}  
|I_1|\le c \|(-X\cdot\nabla_Y+\partial_t)f \|_{L^2(V_{Y,t},H_X^{-1}(U_X))}.
\end{equation*}
Using \eqref{e.bounds.phi.kin} and the fact that $\xi_0$ has compact support,
\begin{equation*}  
|I_2|\leq c\|f - \la f \ra_{U_X} \|_{L^2(V_{Y,t},L^2(U_X))} \leq c\|\nabla_X f\|_{L^2(V_{Y,t},L^2(U_X))}.
\end{equation*}
This completes the proof of the estimate in \eqref{e.bracket.H-1.kin} involving the time derivative. To estimate the terms involving $\partial_{y_i}$, we fix $i \in \{1,\ldots,m\}$ and we choose a smooth function $\xi_i \in C^\infty_0(B(0,2R))$ satisfying
\begin{equation*}  
\int_{B(0,2R)} \xi_i(X) \, \d X = 0 \quad \text{ and } \quad \int_{B(0,2R)} x_j \xi_i(X) \, \d X = \delta_{ij},
\end{equation*}
where $\delta_{ij}$ is the Kronecker delta. Using this we see that
\begin{align*}  
 &\iint_{V_{Y,t}} \phi \, \partial_{y_i} \la f \ra_{U_X}\, \d Y\d t\\
  &= \iiint_{U_X\times V_{Y,t}} \xi_i(X) \bigl(X \cdot \nabla_Y \phi(Y,t)- \partial_t \phi(Y,t)\bigr) \la f \ra_{U_X}(Y,t) \,  \d X\d Y\d t\\
  &=\iiint_{U_X\times V_{Y,t}} \xi_i(X) \bigl(X \cdot \nabla_Y \phi(Y,t)- \partial_t \phi(Y,t)\bigr) f(X,Y,t) \,   \d X\d Y\d t\\
  &\quad +\iiint_{U_X\times V_{Y,t}} \xi_i(X) \bigl(X \cdot \nabla_Y \phi(Y,t)- \partial_t \phi(Y,t)\bigr) (\la f \ra_{U_X}(Y,t)-f(X,Y,t)) \,   \d X\d Y\d t.
\end{align*}
Proceeding as above we deduce that
\begin{align*}  
 \Ll|\iint_{V_{Y,t}} \phi \, \partial_{y_i} \la f \ra_{U_X}\, \d Y\d t\Rr|\leq c\bigl (\|\nabla_X f\|_{L^2(V_{Y,t},L^2(U_X))}+
 \|(-X\cdot\nabla_Y+\partial_t)f \|_{L^2(V_{Y,t},H_X^{-1}(U_X))}\bigr).
\end{align*}
This completes the proof of \eqref{e.bracket.H-1.kin} and we can conclude that
\begin{equation}\label{e.duale}
\begin{split}
&||\partial_{t} \la f\ra_{U_X}||_{H^{-1}(V_{Y,t})}+||\nabla_{Y} \la f\ra_{U_X}||_{H^{-1}(V_{Y,t})}\\
&\leq c\bigl(\|\nabla_X f\|_{L^2(V_{Y,t},L^2(U_X))}+
 \|(-X\cdot\nabla_Y+\partial_t)f \|_{L^2(V_{Y,t},H_X^{-1}(U_X))}\bigr).
 \end{split}
\end{equation}

Using \eqref{e.duale}, and Lemma 3.1 in \cite{AM}, we see that
$$\| f - \la f \ra_{U_X\times V_{Y,t}} \|_{L_{Y,t}^2(V_{Y,t},L^2(U_X))}$$
can be estimated from above by
\begin{align*} \label{}
&
\| f - \left\langle f \right\rangle_{U_X} \|_{L_{Y,t}^2(V_{Y,t},L^2(U_X))}
+ \| \left\langle f \right\rangle_{U_X} - \left\langle f \right\rangle_{U_X\times V_{Y,t}} \|_{L^2(V_{Y,t})}
\\ &
\leq
\| f - \left\langle f \right\rangle_{U_X} \|_{L_{Y,t}^2(V_{Y,t},L^2(U_X))}+c\| \partial_t \left\langle f \right\rangle_{U_X}  \|_{H^{-1}(V_{Y,t})}
+ c\| \nabla_Y \left\langle f \right\rangle_{U_X}  \|_{H^{-1}(V_{Y,t})}
\\ &
\leq
c\bigl(\|\nabla_X f\|_{L^2(V_{Y,t},L^2(U_X))}+
 \|(-X\cdot\nabla_Y+\partial_t)f \|_{L^2(V_{Y,t},H_X^{-1}(U_X))}\bigr).
\end{align*}
In particular, we have proved that
\begin{equation} \label{e.poincare.mean}
\begin{split}
&\| f - \left\langle f \right\rangle_{U_X\times V_{Y,t}} \|_{L_{Y,t}^2(V_{Y,t},L^2(U_X))}\\
&\leq
c\bigl(\|\nabla_X f\|_{L^2(V_{Y,t},L^2(U_X))}+
 \|(-X\cdot\nabla_Y+\partial_t)f \|_{L^2(V_{Y,t},H_X^{-1}(U_X))}\bigr).
 \end{split}
\end{equation}

Note that the above deductions hold for all $f\in  W(U_X\times V_{Y,t})\cap C^\infty(\overline{U_X\times V_{Y,t}})$. Next, to complete the proof of the lemma we consider $f\in  W_{Y,t,0}(U_X\times V_{Y,t})\cap C^\infty(\overline{U_X\times V_{Y,t}})$ and we will use the additional information encoded in
$W_{Y,t,0}(U_X\times V_{Y,t})$ to complete the proof. In fact we will prove that if $f\in  W_{Y,t,0}(U_X\times V_{Y,t})\cap C^\infty(\overline{U_X\times V_{Y,t}})$, then
\begin{equation}
\label{e.meancontrol}
\left| \left\langle f \right\rangle_{U_X\times V_{Y,t}} \right|
\leq
c\bigl(\|\nabla_X f\|_{L^2(V_{Y,t},L^2(U_X))}+
 \|(-X\cdot\nabla_Y+\partial_t)f \|_{L^2(V_{Y,t},H_X^{-1}(U_X))}\bigr).
\end{equation}

Consider $\tilde f_1 \in C^\infty(\overline{U_X\times V_{Y,t}})$ and let $\phi \in C^\infty_0(B(0,2R))$ be such that $\phi\equiv 1$ on $B(0,R)$, $\phi\geq 0$. 
Let $f_1(X,Y,t)=\phi(X)\tilde f_1(X,Y,t)$. Observe that by the divergence theorem
\begin{align*} \label{}
\bariiint_{U_X\times V_{Y,t}}(X\cdot \nabla_Y-\partial_t) f_1\, \d X\d Y\d t
=  \frac{1}{|{U_X\times V_{Y,t}}|} \int_{\partial {V_{Y,t}}} \int_{B(0,2R)} (X,-1)\cdot N_{Y,t} f_1\, \d X\d\sigma_{Y,t},
\end{align*}
where $\sigma_{Y,t}$ is the surface measure on $\partial {V_{Y,t}}$. We now choose a non-trivial $f_1$ so that
$f_1=0$ on $\{(X,Y,t)\in B(0,2R)\times \partial V_{Y,t}\mid (X,-1)\cdot N_{Y,t}\leq 0\}$ and so that, after a normalization,
\begin{align} \label{e.testw.meanone}
\bariiint_{U_X\times V_{Y,t}}(X\cdot \nabla_Y-\partial_t) f_1\, \d X\d Y\d t=1,
\end{align}
and
\begin{equation}
\label{e.testw.L2bounds}
\left\|   (X\cdot \nabla_Y-\partial_t)f_1 \right\|_{L^2(V_{Y,t},L^2(U_X))} \leq c.
\end{equation}
Note that, by construction, $ff_1=0$ on $B(0,2R)\times\partial V_{Y,t}$.

To proceed, using $f_1$ and \eqref{e.testw.meanone} we split the mean of~$f$ as
\begin{equation}\label{decomp}
\begin{split}
\left\langle f \right\rangle_{U_X\times V_{Y,t}}&=\bariiint_{U_X\times V_{Y,t}}f (X\cdot \nabla_Y-\partial_t) f_1 \, \d X\d Y\d t\\
&\quad -\bariiint_{U_X\times V_{Y,t}}\bigl( f - \left\langle f \right\rangle_{U_X\times V_{Y,t}}\bigr) (X\cdot \nabla_Y-\partial_t) f_1 \d X\d Y\d t,
\end{split}
\end{equation}
and we estimate the two terms on the right side separately. For the first term we have
\begin{align*}
&\left|
\bariiint_{U_X\times V_{Y,t}}f (X\cdot \nabla_Y-\partial_t) f_1 \, \d X\d Y\d t
\right|\\
&=
\left| -\bariiint_{U_X\times V_{Y,t}}f_1 (X\cdot \nabla_Y-\partial_t) f \, \d X\d Y\d t\right|\\
&\quad + \frac{1}{|U_X\times V_{Y,t}|} \left |\int_{\partial V_{Y,t}} \int_{B(0,2R)} (X,-1)\cdot N_{Y,t} ff_1\, \d X\d\sigma_{Y,t} \right|
\\ &
= \left|\bariiint_{U_X\times V_{Y,t}}f_1 (X\cdot \nabla_Y-\partial_t) f \, \d X\d Y\d t\right|,
\end{align*}
where we used that $(X,-1)\cdot N_{Y,t} ff_1$ vanishes on $B(0,2R)\times\partial V_{Y,t}$ to remove the boundary integral.  Note that
\begin{align*} \label{}
&\left| \bariiint_{U_X\times V_{Y,t}}f_1 (X\cdot \nabla_Y-\partial_t) f \, \d X\d Y\d t\right|\\
&\leq
\frac1{|U_X\times V_{Y,t}|}
 \left\| f_1 \right\|_{L_{Y,t}^2(V_{Y,t},H^1_X(U_X))}
\left\| (X\cdot \nabla_Y-\partial_t) f\right\|_{L_{Y,t}^2(V_{Y,t},H^{-1}_X(U_X))}.
\end{align*}
This completes the estimate for the first term in \eqref{decomp}. For the second term in \eqref{decomp}, we use \eqref{e.testw.L2bounds} to derive
\begin{align*}
&\left|
\bariiint_{U_X\times V_{Y,t}}\bigl( f - \left\langle f \right\rangle_{U_X\times V_{Y,t}}\bigr) (X\cdot \nabla_Y-\partial_t) f_1 \d X\d Y\d t
\right|\\
&\leq
\| f - \left\langle f \right\rangle_{U_X\times V_{Y,t}} \|_{L^2_{Y,t}(V_{Y,t},L^2(U_X))}
\|  (X\cdot \nabla_Y-\partial_t) f_1 \|_{L^2_{Y,t}(V_{Y,t},L^2(U_X))}
\\
&
\leq c\| f - \left\langle f \right\rangle_{U_X\times V_{Y,t}} \|_{L^2_{Y,t}(V_{Y,t},L^2(U_X))}.
\end{align*}
The last term is then estimated using \eqref{e.poincare.mean}.
{This finishes the proof for $f\in  W_{Y,t,0}(U_X\times V_{Y,t})\cap C^\infty(\overline{U_X\times V_{Y,t}})$. Finally, we recall that $W_{Y,t,0}(U_X\times V_{Y,t})$ is defined as the closure of $C^\infty_{Y,t,0}(\overline{U_X\times V_{Y,t}})$ in the norm of $W(U_X\times V_{Y,t})$, and hence a standard density argument finishes the proof for $f\in  W_{Y,t,0}(U_X\times V_{Y,t})$.}

\end{proof}

\section{Proof of Theorem \ref{weakdp1}}\label{Sec2}

The purpose of this section is to prove Theorem \ref{weakdp1}. We therefore fix $U_X\times V_{Y,t}\subset\mathbb R^{N+1}$, $g^*$ and $g$ as in the statement of Theorem \ref{weakdp1}.  To ease the notation we will at instances use the notation
$$W:=W(U_X\times V_{Y,t}),\ W_0:=W_{0}(U_X\times V_{Y,t}),\ W_{Y,t,0}:=W_{Y,t,0}(U_X\times V_{Y,t}).$$

Given an arbitrary pair $(f,f^\ast)$ such that \begin{eqnarray}\label{weak1-a}
f\in L_{Y,t}^2(V_{Y,t},H_X^1(U_X))\mbox{ and } f^\ast-(X\cdot\nabla_Y-\partial_t)f\in  L_{Y,t}^2(V_{Y,t},H_X^{-1}(U_X)),
    \end{eqnarray}
     we set
     \begin{equation}
\label{e.def.J}
J[f,f^*] := \inf \iiint_{U_X\times V_{Y,t}} \frac 1 2 (A(\nabla_X f - \g))\cdot (\nabla_X f - \g) \, \d X\d Y\d t,
\end{equation}
where the infimum is taken with respect to the set
\begin{equation}
\label{e.def.J+}
\bigl\{ \g  \in (L^2(V_{Y,t},L^2(U_X)))^m \mid {\nabla_X \cdot (A\g)} = f^* - (X\cdot\nabla_Y-\partial_t)f \bigr\}.
\end{equation}
The condition
\begin{equation*}  
{\nabla_X \cdot (A\g)}  = f^* - (X\cdot\nabla_Y-\partial_t)f,
\end{equation*}
appearing in \eqref{e.def.J+}, should be interpreted as stating that
\begin{equation}
\label{e.g.condition}
- \iiint_{U_X\times V_{Y,t}} A\g  \cdot \nabla_X \phi \, \d X\d Y\d t = \iint_{V_{Y,t}} \langle f^*(\cdot,Y,t) - (X\cdot\nabla_Y-\partial_t)f(\cdot,Y,t), \phi\rangle \, \d Y\d t,
\end{equation}
for all $\phi \in L^2(V_{Y,t},H^1_{X,0}(U_X))$.

We intend to prove, given $(g,g^*)\in  W(U_X\times V_{Y,t})\times L_{Y,t}^2(V_{Y,t},{H}_X^{-1}(U_X))$, that the functional
\begin{equation*}  \label{prob}
 (g+W_{0})\owns f \mapsto J[f,g^*],
\end{equation*}
is uniformly convex, that its minimum is zero,  and that the associated minimizer is the unique $f \in (g+W_{0})$ that solves the equation
$\L  f=g^*$ in $U_X\times V_{Y,t}$ in the sense that
\begin{equation*}
    \begin{split}
     0 =&\iiint_{U_X\times V_{Y,t}}\ A(X,Y,t)\nabla_Xf\cdot \nabla_X\phi\, \d X \d Y \d t\\
    &+\iint_{V_{Y,t}}\ \langle g^*(\cdot,Y,t) - (X\cdot\nabla_Y-\partial_t)f(\cdot,Y,t),\phi(\cdot,Y,t)\rangle\, \d Y \d t,
\end{split}
\end{equation*}
for all $\phi\in L_{Y,t}^2(V_{Y,t},H_{X,0}^1(U_X))$. Note that by construction this means that $f\in W(U_X\times V_{Y,t})$ is a solution to the problem in \eqref{dpweak} in the sense of Definition \ref{defwe}. In particular, $f-g\in W_0(U_X\times V_{Y,t})$ holds as desired.

\begin{lemma}\label{1stlemma}
    Let $(g,g^*)\in  W(U_X\times V_{Y,t})\times L_{Y,t}^2(V_{Y,t},{H}_X^{-1}(U_X))$ be fixed and let $\A(g,g^\ast)$ be the set
    \begin{equation*}   \{ (f, \j) \in (g + W_{0})\times( L^2(V_{Y,t},L^2(U_X)))^m \mid \nabla_X\cdot(A(X,Y,t)\j) = g^\ast- (X\cdot\nabla_Y-\partial_t)f \}.
\end{equation*}
Then $\A(g,g^\ast)$ is non-empty.
\end{lemma}
\begin{proof}
Let $f=g$  and consider the equation
\begin{align}\label{ex1}
&{\nabla_X \cdot (A(X,Y,t)\nabla_Xv(X,Y,t))} = (g^\ast(X,Y,t)- (X\cdot\nabla_Y-\partial_t)f(X,Y,t))\in H_X^{-1}(U_X),
\end{align}
for $\d Y\d t$-a.e $(Y,t)\in V_{Y,t}$. By the Lax-Milgram theorem this equation has a (unique) solution $v(\cdot)=v(\cdot,Y,t)\in H^1_{X,0}(U_X)$ and
\begin{align}  \label{ex2}
 ||\nabla_Xv||_{L_{Y,t}^2(V_{Y,t},L^2(U_X))}\leq c||g^\ast-(X\cdot\nabla_Y-\partial_t)f||_{L_{Y,t}^2(V_{Y,t},{H}_X^{-1}(U_X))}<\infty,
\end{align}
as $f\in W$. In particular,
\begin{equation}\label{nempty}
(g,\nabla_Xv)\in \A(g,g^\ast),
\end{equation} and hence $\A(g,g^\ast)$ is non-empty.
\end{proof}


\begin{lemma} For every $f \in (g+W_0)$ and $\j \in  (L^2(V_{Y,t},L^2(U_X)))^m$, define
\begin{equation*}  
\J[f,\j] := \iiint_{U_X\times V_{Y,t}} \frac 1 2 (A(\nabla_X f - \j))\cdot (\nabla_X f - \j) \, \d X\d Y\d t.
\end{equation*}
Then the functional $\J$ is uniformly convex on $\A(g,g^*)$.
\end{lemma}
\begin{proof} By Lemma \ref{1stlemma} we know that $\A(g,g^*)\neq \emptyset$. A straightforward calculation shows that
\begin{equation*}  
\frac 1 2 \J[f' + f, \j' + \j] + \frac 1 2\J[f'-f, \j' - \j] - \J[f',\j'] = \J[f,\j],
\end{equation*}
whenever $(f',\j') \in \A(g,g^*)$ and $(f,\j) \in \A(0,0)$. Hence, to prove the lemma it suffices to prove that there exists $1\leq c < \infty$,  depending on $m$, $\kappa$, and $U_X\times V_{Y,t}$, such that for every $(f,\j) \in \A(0,0)$ we have
\begin{equation}
\label{e.convex}
 \J[f,\j] \ge c^{-1} \bigl(\|f\|_{W}^2 + \||\j|\|_{L^2(V_{Y,t}, L^2(U_X))}^2 \bigr).
\end{equation}
Considering $(f,\j) \in \A(0,0)$, {by expanding the integrand in the definition of $\J[f,\j]$, using symmetry of $A$, and using that} $$\nabla_X\cdot(A(X,Y,t) \j) = - (X\cdot\nabla_Y-\partial_t)f,$$ we find
\begin{align*}  
\J[f,\j] &= \iiint_{U_X\times V_{Y,t}} \bigl(\frac 1 2 A\nabla_X f\cdot \nabla_X f + \frac 1 2 A\j\cdot\j\bigr) \, \d X\d Y\d t\notag\\
&\quad - \iint_{V_{Y,t}} \langle(X\cdot\nabla_Y-\partial_t)f(\cdot,Y,t),f(\cdot,Y,t)\rangle \, \d Y\d t.
\end{align*}
If $(f,\j) \in \A(0,0)$, then $f\in W_{0}$.  Recall that $W_0=W_0(U_X\times V_{Y,t})$ is the closure in the norm of
$W(U_X\times V_{Y,t})$ of $C^\infty_{\K,0}(\overline{U_X\times V_{Y,t}})$. In particular, there exists $\{f_j\}$, $f_j\in C^\infty_{\K,0}(\overline{U_X\times V_{Y,t}})$ such that
$$||f-f_j||_W\to 0\mbox{ as }j\to \infty,$$
and consequently
$$||(X\cdot\nabla_Y-\partial_t)(f-f_j)||_{L_{Y,t}^2(V_{Y,t},{H}_X^{-1}(U_X))}\to 0\mbox{ as }j\to \infty.$$
Using this we see that
\begin{equation}  \label{apa}
\begin{split}
&\iint_{V_{Y,t}} \langle(X\cdot\nabla_Y-\partial_t)f(\cdot,Y,t),f(\cdot,Y,t)\rangle \, \d Y\d t\\
&\leq \limsup_{j\to\infty}
\iint_{V_{Y,t}} \langle(X\cdot\nabla_Y-\partial_t)f_j(\cdot,Y,t),f_j(\cdot,Y,t)\rangle \, \d Y\d t.
\end{split}
\end{equation}
However, using that $f_i\in C^\infty_{\K,0}(\overline{U_X\times V_{Y,t}})$  we see that
\begin{equation} \label{apa+}
\begin{split}
&\iint_{V_{Y,t}} \langle(X\cdot\nabla_Y-\partial_t)f_j(\cdot,Y,t),f_j(\cdot,Y,t)\rangle \, \d Y\d t\\
&=\iiint_{U_X\times V_{Y,t}} (X\cdot\nabla_Y-\partial_t)f_j f_j \, \d X\d Y\d t\\
&=\frac 1 2\iiint_{U_X\times V_{Y,t}} (X\cdot\nabla_Y-\partial_t)f_j^2 \, \d X\d Y\d t\\
&=\frac 12 \int_{U_X}\iint_{\partial V_{Y,t}} f_j^2(X,-1)\cdot N_{Y,t} \, \d \sigma_{Y,t}\d X\leq 0,
\end{split}
\end{equation}
by the divergence theorem and the definition of the Kolmogorov boundary. Hence we can conclude that
\begin{align*}  
\J[f,\j] \geq c^{-1} \iiint_{U_X\times V_{Y,t}} \bigl(|\nabla_X f|^2 +|\j|^2\bigr) \, \d X\d Y\d t,
\end{align*}
for all $(f,\j) \in \A(0,0)$. To prove  \eqref{e.convex}, and hence that $\J$ is uniformly convex on $\A(g,g^*)$, it therefore suffices to prove that
\begin{align*}  
&\iiint_{U_X\times V_{Y,t}} |f|^2 \, \d X\d Y\d t+ \iint_{V_{Y,t}}||(-X\cdot\nabla_Y+\partial_t)f(\cdot,Y,t)||_{{H}_X^{-1}(U_X)}^2\, \d Y\d t\notag\\
&\leq c \iiint_{U_X\times V_{Y,t}} \bigl(|\nabla_X f|^2 +  |\j|^2\bigr) \, \d X\d Y\d t,
\end{align*}
whenever $(f,\j) \in \A(0,0)$. However, using that $\nabla_X\cdot(A(X,Y,t)\j) = - (X\cdot\nabla_Y-\partial_t)f$ we immediately see that
\begin{align*}  
\iint_{V_{Y,t}}||(-X\cdot\nabla_Y+\partial_t)f(\cdot,Y,t)||_{{H}_X^{-1}(U_X)}^2\, \d Y\d t\leq c \iiint_{U_X\times V_{Y,t}}  |\j|^2 \, \d X\d Y\d t,
\end{align*}
by the boundedness of $A$.  Furthermore, using the  Poincar\'e inequality of Lemma \ref{lemma1a-}, or Lemma \ref{lemma1a}, we can conclude that
\begin{align*}  
c^{-1}\iiint_{U_X\times V_{Y,t}} |f|^2 \, \d X\d Y\d t\leq& \iiint_{U_X\times V_{Y,t}}|\nabla_X f|^2 \, \d X\d Y\d t\\
&+\iint_{V_{Y,t}}||(-X\cdot\nabla_Y+\partial_t)f(\cdot,Y,t)||_{{H}_X^{-1}(U_X)}^2\, \d Y\d t,
\end{align*}
for all $f\in W_{0}$. Put together this yields \eqref{e.convex} and the proof of the lemma is complete.
\end{proof}

As the functional $\J$ is uniformly convex over $\A(g,g^*)$ there exist a unique minimizing pair $(f_1,\j_1)\in \A(g,g^*)$ such that
\begin{equation*}  
(f_1,\j_1):=\argmin_{(f,\j)\in \A(g,g^*)} \J[f,\j] =\argmin_{(f,\j)\in \A(g,g^*)} \iiint_{U_X\times V_{Y,t}} \frac 1 2 (A(\nabla_X f - \j))\cdot (\nabla_X f - \j) \, \d X\d Y\d t.
\end{equation*}
Note that, by construction and by the ellipticity of $A$, we have
\begin{equation*}  
J[f_1,g^*] \ge 0.
\end{equation*}

\begin{lemma} There is a one-to-one correspondence between weak solutions to $\L  u=g^*$ in $U_X\times V_{Y,t}$, such that $(u-g)\in W_0$,  and null minimizers of $J[\cdot,g^*]$.
\end{lemma}
\begin{proof} To prove the lemma we need to prove that for every $f \in g + W_{0}$, we have
\begin{align*}  
f \mbox{ solves $\L  u=g^*$ in the weak sense in $U_X\times V_{Y,t}$} \iff J[f,g^*] = 0.
\end{align*}
Indeed, the implication ''$\implies$'' is clear since if $f$ solves $\L  u=g^*$ in the weak sense, then
\begin{equation*}  
(f,\nabla_X f) \in \A(g,g^*) \ \text{ and } \ \J[f,\nabla_X f] = 0=J[f,g^*].
\end{equation*}
Conversely, if $J[f,g^*] = 0$, then $f = f_1$ and $\J[f_1,\j_1] = 0$. This implies that
\begin{equation*}  
\nabla_X f_1 = \j_1 \qquad \mbox{a.e. in } U_X\times V_{Y,t},
\end{equation*}
and since $\nabla_X\cdot(A\j_1) = g^*-(X\cdot\nabla_Y-\partial_t)f_1$, we recover that $f = f_1$ is indeed a weak solution of $\L  u=g^*$. In particular, the fact that there is at most one solution to $\L  u=g^*$ is clear.
\end{proof}

To complete the proof of Theorem \ref{weakdp1} it remains to prove that
\begin{equation}
\label{e.null.min}
J[f_1,g^*] \le 0.
\end{equation}
In order to do so, we introduce the perturbed convex minimization problem defined, for every $f^* \in L_{Y,t}^2(V_{Y,t},H_X^{-1}(U_X))$, by
\begin{equation*}  
G(f^*) := \inf_{f \in W_{0}}\bigl ( J[f + g, f^*+g^* ]  {- \iint_{ V_{Y,t}} \langle f^*(\cdot,Y,t),f(\cdot,Y,t)\rangle\, \d Y\d t\bigr ).}
\end{equation*}
As \begin{equation*}  
G(0) = \inf_{f \in W_{0}} J[f + g, g^* ],
\end{equation*}
we see that the inequality in \eqref{e.null.min} that we intend to prove can be equivalently stated as $G(0) \le 0$.

We first intend to verify that the function $G$ is convex and then reduce the problem of showing \eqref{e.null.min} to that of showing that the convex dual of $G$ is non-negative.

\begin{lemma} $G$ is a convex, locally bounded from above and lower semi-continuous functional on $L_{Y,t}^2(V_{Y,t},H_X^{-1}(U_X))$.
\end{lemma}
\begin{proof} For every pair $(f,\j)$ satisfying $(f+g,\j) \in \mcl A(g,f^*+g^*)$, we have
\begin{equation*}
\nabla_X \cdot (A \j) = f^*+g^* - (X\cdot\nabla_Y-\partial_t)(f + g),
\end{equation*}
and thus
\begin{align*}  
\J[f+g,\j] & = \iiint_{U_X\times V_{Y,t}} \frac 1 2 (A(\nabla_X (f+g)-\j))\cdot (\nabla_X (f+g)-\j) \, \d X\d Y\d t \\
& =\iiint_{U_X\times V_{Y,t}}  \frac 1 2 (A\nabla_X (f+g))\cdot(\nabla_X (f+g))  + \frac 1 2 A\j\cdot\j\, \d X\d Y\d t\\
&\quad -\iint_{V_{Y,t}}   \langle (X\cdot\nabla_Y-\partial_t)(f+g)(\cdot,Y,t),(f+g)(\cdot,Y,t)\rangle\, \d Y\d t\notag\\
&\quad {+\iint_{V_{Y,t}}   \langle (f^*+g^*)(\cdot,Y,t),(f+g)(\cdot,Y,t)\rangle\, \d Y\d t.}
\end{align*}
In particular, {subtracting  $$\iint_{V_{Y,t}} \langle f^*(\cdot,Y,t),f(\cdot,Y,t)\rangle\, \d Y\d t$$ from} the expression above we see that
\begin{align*}  
&\J[f+g,\j] {-\iint_{V_{Y,t}} \langle f^*(\cdot,Y,t),f(\cdot,Y,t)\rangle\, \d Y\d t} \notag\\
& =\iiint_{U_X\times V_{Y,t}}  \frac 1 2 (A\nabla_X (f+g))\cdot(\nabla_X (f+g))  + \frac 1 2 A\j\cdot\j\, \d X\d Y\d t\\
&\quad -\iint_{V_{Y,t}}   \langle (X\cdot\nabla_Y-\partial_t)(f+g)(\cdot,Y,t),(f+g)(\cdot,Y,t)\rangle\, \d Y\d t\notag\\
&\quad {+\iint_{V_{Y,t}}   \langle f^*(\cdot,Y,t),g(\cdot,Y,t)\rangle+\langle g^*(\cdot,Y,t),(f+g)(\cdot,Y,t)\rangle\, \d Y\d t.}
\end{align*}
Taking the infimum over all $(f,\j)$ satisfying the affine constraint $(f+g,\j) \in \mcl A(g,f^*+g^*)$ we obtain the quantity $G(f^*)$, i.e., $G(f^*)$ can be expressed as
\begin{equation*}  
G(f^*) = \inf_{(f,\j):\  (f+g,\j) \in \mcl A(g,f^*+g^*)}\bigl ( \J[f+g,\j] { - \iint_{V_{Y,t}} \langle f^*(\cdot,Y,t),f(\cdot,Y,t)\rangle\, \d Y\d t} \bigr ).
\end{equation*}
Next, writing
\begin{align*}
&-\iint_{V_{Y,t}}   \langle (X\cdot\nabla_Y-\partial_t)(f+g)(\cdot,Y,t),(f+g)(\cdot,Y,t)\rangle \, \d Y\d t\notag\\
&=-\iint_{V_{Y,t}}   \langle (X\cdot\nabla_Y-\partial_t)f(\cdot,Y,t),f(\cdot,Y,t)\rangle \, \d Y\d t\notag\\
&\quad-\iint_{V_{Y,t}}   \langle (X\cdot\nabla_Y-\partial_t)f(\cdot,Y,t),g(\cdot,Y,t)\rangle \, \d Y\d t\notag\\
&\quad-\iint_{V_{Y,t}}   \langle (X\cdot\nabla_Y-\partial_t)g(\cdot,Y,t),(f+g)(\cdot,Y,t)\rangle \, \d Y\d t,
\end{align*}
we see that we can argue as in \eqref{apa} and \eqref{apa+} to conclude that
\begin{align}  \label{ex3}
&-\iint_{V_{Y,t}}   \langle (X\cdot\nabla_Y-\partial_t)f(\cdot,Y,t),f(\cdot,Y,t)\rangle \, \d Y\d t\notag\\
&=-\frac 12\iint_{U_X\times \partial V_{Y,t}}  f^2(X,Y,t) \, (X,-1)\cdot N_{Y,t}\d\sigma_{Y,t}\d X\\
&=\frac 12\iint_{U_X\times \partial V_{Y,t}}  f^2(X,Y,t) \, |(X,-1)\cdot N_{Y,t}|\d\sigma_{Y,t}\d X,\notag
\end{align}
as $f\in W_0$. In particular, $G(f^*)$ can be expressed as the infimum of
\begin{align*}  
&\iiint_{U_X\times V_{Y,t}}  \frac 1 2 (A\nabla_X (f+g))\cdot(\nabla_X (f+g))  + \frac 1 2 A\j\cdot\j\, \d X\d Y\d t\\
&+\frac 12\iint_{U_X\times \partial V_{Y,t}}  f^2(X,Y,t) \, |(X,-1)\cdot N_{Y,t}|\d\sigma_{Y,t}\d X\notag\\
&-\iint_{V_{Y,t}}   \langle (X\cdot\nabla_Y-\partial_t)f(\cdot,Y,t),g(\cdot,Y,t)\rangle \, \d Y\d t\notag\\
&-\iint_{V_{Y,t}}   \langle (X\cdot\nabla_Y-\partial_t)g(\cdot,Y,t),(f+g)(\cdot,Y,t)\rangle \, \d Y\d t\notag\\
&{+\iint_{V_{Y,t}}   \langle f^*(\cdot,Y,t),g(\cdot,Y,t)\rangle+\langle g^*(\cdot,Y,t),(f+g)(\cdot,Y,t)\rangle\, \d Y\d t,}
\end{align*}
with respect to $(f,\j)$ such that $(f+g,\j) \in \mcl A(g,f^*+g^*)$. The expression in the last display is convex as a function of $(f,f^*,\j)$ and this proves that $G$ is convex. Furthermore, using \eqref{ex1}, \eqref{ex2}, \eqref{nempty}, and \eqref{ex3} we can conclude that the {infimum of the expression in} the last display is finite, hence $G(f^*)<\infty$. In particular, the function $G$ is locally bounded from above. These two properties imply that $G$ is lower semi-continuous, see \cite[Lemma~I.2.1 and Corollary~I.2.2]{ET}.
\end{proof}

 We denote by $G^*$ the convex dual of $G$, defined for every
$$h \in (L_{Y,t}^2(V_{Y,t},H_X^{-1}(U_X)))^\ast=L_{Y,t}^2(V_{Y,t},H_{X,0}^1(U_X)),$$ as
\begin{equation*}  
G^*(h) := \sup_{f^* \in L_{Y,t}^2(V_{Y,t},H_X^{-1}(U_X))} \bigl( -G(f^*) + \iint_{V_{Y,t}} \langle f^*(\cdot,Y,t),h(\cdot,Y,t)\rangle\, \d Y\d t \bigr).
\end{equation*}
Let $G^{**}$ be the bidual of $G$. Since $G$ is lower semi-continuous, we have that $G^{**} = G$ (see \cite[Proposition~I.4.1]{ET}), and in particular,
\begin{equation*}  
G(0) = G^{**}(0) = \sup_{h \in L_{Y,t}^2(V_{Y,t},H_{X,0}^1(U_X))} \bigl( -G^*(h) \bigr) .
\end{equation*}
In order to prove that $G(0) \le 0$, it therefore suffices to show that
\begin{equation}
\label{e.get.to.nullmin}
G^*(h) \ge 0\mbox{ for all } h \in L_{Y,t}^2(V_{Y,t},H_{X,0}^1(U_X)).
\end{equation}

To continue we note that we can rewrite $G^*(h)$ as
\begin{equation}
\label{e.second.G*}
\begin{split}
G^*(h) = \sup_{(f,\j,f^*)} &\bigg\{ \iiint_{U_X\times V_{Y,t}} -\frac 1 2 (A(\nabla_X (f+g)-\j))\cdot (\nabla_X (f+g)-\j) \, \d X\d Y\d t\\
&+\iint_{V_{Y,t}} {\langle f^*(\cdot,Y,t),(h(\cdot,Y,t)+f(\cdot,Y,t))\rangle}\, \d Y\d t\bigg \},
\end{split}
\end{equation}
where the supremum is taken with respect to $$(f,\j,f^*)\in W_{0}\times (L_{Y,t}^2(V_{Y,t},L^2_X(U_X)))^m\times L_{Y,t}^2(V_{Y,t},H_X^{-1}(U_X)),$$ subject to the constraint
\begin{equation}
\label{e.constraint.fj}
\nabla_X\cdot (A \j) = f^* + g^* - (X\cdot\nabla_Y-\partial_t)(f + g).
\end{equation}
Furthermore, note that  for every $h \in  L_{Y,t}^2(V_{Y,t},H_{X,0}^1(U_X))$, we have $G^*(h) \in \R \cup \{+\infty\}$.

\begin{lemma}\label{red} Consider $h \in  L_{Y,t}^2(V_{Y,t},H_{X,0}^1(U_X))$. Then
\begin{equation}
\label{e.finite.G*}
G^*(h) < +\infty \quad \implies \quad h \in W\cap L_{Y,t}^2(V_{Y,t},H_{X,0}^1(U_X)).
\end{equation}
\end{lemma}
\begin{proof} To prove the lemma we need to prove that $(-X\cdot\nabla_Y+\partial_t)h \in L_{Y,t}^2(V_{Y,t},H_{X}^{-1}(U_X))$.  Using that we take a supremum in the definition of $G^\ast$ we can develop lower bounds on $G^\ast$ by restricting the set with respect to which we take the supremum. Here, for $f\in W_0$, we choose to restrict the supremum  to $(f,\j,f^*)$ where  $\j=\j_0$ is a solution of $\nabla_X\cdot(A \j_0)= g^* -(X\cdot\nabla_Y-\partial_t)g$ and $f^* := (X\cdot\nabla_Y-\partial_t)f$. Recall from \eqref{nempty} that such a $\j_0 \in (L_{Y,t}^2(V_{Y,t},L^2_X(U_X)))^m$ exists. With these choices for $\j$ and  $f^*$, the constraint \eqref{e.constraint.fj} is satisfied, and we obtain that
\begin{align*}  
G^*(h)\geq \sup_{f\in W_0} &\biggl\{ \iiint_{U_X\times V_{Y,t}} -\frac 1 2 (A(\nabla_X (f+g)-\j_0))\cdot (\nabla_X (f+g)-\j_0) \, \d X\d Y\d t\notag\\
&+\iint_{V_{Y,t}} {\langle (X\cdot\nabla_Y-\partial_t)f(\cdot,Y,t),(h(\cdot,Y,t) + f(\cdot,Y,t))\rangle}\, \d Y\d t\biggr \}.
\end{align*}
Consider  $f\in  C^\infty_{\K,0}(\overline{U_X\times V_{Y,t}})\subset W_0$. Then, again arguing as in \eqref{apa}, \eqref{apa+},
\begin{align*}  
\iint_{V_{Y,t}} \langle (X\cdot\nabla_Y-\partial_t)f(\cdot,Y,t),f(\cdot,Y,t)\rangle\, \d Y\d t\leq 0.
\end{align*}
{Furthermore, restricting to $f\in C^\infty_{0}({U_X\times V_{Y,t}})\subset C^\infty_{\K,0}(\overline{U_X\times V_{Y,t}})$ yields by the same argument that
\begin{align*}  
\iint_{V_{Y,t}} \langle (X\cdot\nabla_Y-\partial_t)f(\cdot,Y,t),f(\cdot,Y,t)\rangle\, \d Y\d t = 0.
\end{align*}
}
Hence we have the lower bound
\begin{align*}  
G^*(h)\geq \sup &\biggl\{ \iiint_{U_X\times V_{Y,t}} -\frac 1 2 (A(\nabla_X (f+g)-\j_0))\cdot (\nabla_X (f+g)-\j_0) \, \d X\d Y\d t\notag\\
&+\iint_{V_{Y,t}} \langle (X\cdot\nabla_Y-\partial_t)f(\cdot,Y,t),h(\cdot,Y,t)\rangle\, \d Y\d t\biggr \},
\end{align*}
where the supremum now is with respect to $f \in  C^\infty_{0}({U_X\times V_{Y,t}})\subset C^\infty_{\K,0}(\overline{U_X\times V_{Y,t}})$. Moreover, as $G^*(h) < +\infty$, we have that
\begin{align*}  
&\iint_{V_{Y,t}} \langle (X\cdot\nabla_Y-\partial_t)f(\cdot,Y,t),h(\cdot,Y,t)\rangle\, \d Y\d t\notag\\
&\leq  \frac 1 2\iiint_{U_X\times V_{Y,t}} (A(\nabla_X (f+g)-\j_0))\cdot (\nabla_X (f+g)-\j_0) \, \d X\d Y\d t +G^*(h)<\infty,
\end{align*}
for every $f \in  C^\infty_{0}({U_X\times V_{Y,t}})$ fixed. {Note that by replacing $f$ by $-f$ in the above argument we also obtain a lower bound}. In particular,
%
{$$\sup \ \biggl |\iint_{V_{Y,t}} \langle (X\cdot\nabla_Y-\partial_t)h(\cdot,Y,t),f(\cdot,Y,t)\rangle\, \d Y\d t\biggr | <\infty,$$
where the supremum is taken over $f \in  C^\infty_{0}({U_X\times V_{Y,t}})$ such that $||f||_{L_{Y,t}^2(V_{Y,t},H_{X,0}^1(U_X))}\leq 1$.}
 This implies that
$$(-X\cdot\nabla_Y+\partial_t)h\in L_{Y,t}^2(V_{Y,t},H_X^{-1}(U_X))$$ and this observation proves \eqref{e.finite.G*}.
\end{proof}

Lemma \ref{red} gives at hand that  in place of \eqref{e.get.to.nullmin}, we have reduced the matter to proving that
\begin{equation}
\qquad G^*(h) \ge 0\mbox{ for all }h\in W\cap L_{Y,t}^2(V_{Y,t},H_{X,0}^1(U_X)).
\end{equation}
Furthermore, note that for $\tilde h\in W\cap C_{X,0}^\infty(\overline{U_X\times V_{Y,t}})$ we have
\begin{equation}\label{smoothbound}
    G^*(h) \geq G^*(\tilde h) - \|f^*\|_{L^2_{Y,t}(V_{Y,t},H_X^{-1}(U_X))} \| h-\tilde h \|_{L^2_{Y,t}(V_{Y,t},H^1_{X}(U_X))}.
\end{equation}
As we are to establish a lower bound on $G^*$,  we may restrict to taking the supremum over $f^*$ such that
\begin{equation}\label{f*norm}
\|f^*\|_{L_{Y,t}^2(V_{Y,t},H_X^{-1}(U_X))}\leq 1.
\end{equation}
In Lemma \ref{red+} below we prove that
$$G^*(h) \ge 0\mbox{ for all }h\in W\cap C_{X,0}^\infty(\overline{U_X\times V_{Y,t}}).$$
By combining this with \eqref{smoothbound} and \eqref{f*norm} we see that
\[
    G^*(h) \geq G^*(\tilde h) - \| h-\tilde h \|_{L^2_{Y,t}(V_{Y,t},H^1_{X,0}(U_X))}\geq  - \| h-\tilde h \|_{L^2_{Y,t}(V_{Y,t},H^1_{X,0}(U_X))},
\]
for all $\tilde h \in W\cap C^\infty_{X,0}(\overline{U_X\times V_{Y,t}})$. Furthermore, by the definitions of $W$, and $L_{Y,t}^2(V_{Y,t},H_{X,0}^1(U_X))$, we can choose a sequence $h_j\in W\cap C^\infty_{X,0}(\overline{U_X\times V_{Y,t}})$ such that
\[
\lim_{j\rightarrow\infty} \| h-h_j \|_{L^2_{Y,t}(V_{Y,t},H^1_{X,0}(U_X))} = 0.
\]
Hence the proof that $G^*(h)\geq 0$, and hence the final piece in the proof of existence in Theorem \ref{weakdp1} for general $g\in W(U_X\times V_{Y,t})$, is to prove the following lemma.

\begin{lemma}\label{red+}
\begin{equation}
\label{e.get.to.nullmin2}
\qquad G^*(h) \ge 0\mbox{ for all }h\in W\cap C_{X,0}^\infty(\overline{U_X\times V_{Y,t}}).
\end{equation}
\end{lemma}
\begin{proof} To start the proof of the lemma we first note that we have, as $f\in W_0$,  that $$(X\cdot\nabla_Y-\partial_t)f \in L_{Y,t}^2(V_{Y,t},H_X^{-1}(U_X)),$$ and hence we can replace $f^*$ by {$f^* + (X\cdot\nabla_Y-\partial_t)f$} in the variational formula \eqref{e.second.G*} for $G^*$ to get
\begin{align*}
G^*(h) = \sup_{(f,\j,f^*)} &\biggl\{ \iiint_{U_X\times V_{Y,t}} -\frac 1 2 (A(\nabla_X (f+g)-\j))\cdot (\nabla_X (f+g)-\j) \, \d X\d Y\d t\notag\\
&+\iint_{V_{Y,t}} {\langle (f^* + (X\cdot\nabla_Y-\partial_t)f)(\cdot,Y,t),(h(\cdot,Y,t) + f(\cdot,Y,t))\rangle}\, \d Y\d t\biggr \},
\end{align*}
where the supremum now is taken with respect to
\begin{align}
\label{e.rest}(f,\j,f^*)\in (W\cap C^\infty_{\K,0}(\overline{U_X\times V_{Y,t}}))\times (L_{Y,t}^2(V_{Y,t},L^2_X(U_X)))^m\times L_{Y,t}^2(V_{Y,t},H_X^{-1}(U_X)),
\end{align}
subject to the constraint
\begin{equation}
\label{e.constraint.fj2}
\nabla_X\cdot(A \j )= f^* + g^*  - (X\cdot\nabla_Y-\partial_t)g,
\end{equation}
as well as the constraint \eqref{f*norm}. Next  using that $f\in C^\infty_{\K,0}(\overline{U_X\times V_{Y,t}})$, $h\in C^\infty_{X,0}(\overline{U_X\times V_{Y,t}})$, we have
\begin{align*}  
&\iint_{V_{Y,t}} {\langle (X\cdot\nabla_Y-\partial_t)f(\cdot,Y,t),(h(\cdot,Y,t)+f(\cdot,Y,t))\rangle}\, \d Y\d t\\
&=\iint_{V_{Y,t}}{(X\cdot\nabla_Y-\partial_t)f(\cdot,Y,t)(h(\cdot,Y,t)+f(\cdot,Y,t))}\, \d Y\d t\\
&={-\iint_{V_{Y,t}} \langle (X\cdot\nabla_Y-\partial_t)h(\cdot,Y,t),f(\cdot,Y,t)\rangle\, \d Y\d t}\\
&\quad {+\int_{U_X}\iint_{\partial V_{Y,t}} \bigl (\frac 1 2 f^2+fh)(X,-1)\cdot N_{Y,t}\, \d\sigma_{Y,t}\d X.}
\end{align*}
Note that the last integral is not necessarily zero as $f\in W_{0}$ only implies that $f=0$ on $\partial_{\K}(U_X\times V_{Y,t})$ and not necessarily on
$\{(X,Y,t)\in \overline{U_X}\times \partial V_{Y,t}\mid (X,-1)\cdot N_{Y,t}\leq 0\}$.  In any case, using the identity in the last display we see that
\begin{align}
\label{e.G*2a}
G^*(h) = \sup_{(f,\j,f^*)} &\biggl\{ \iiint_{U_X\times V_{Y,t}} \frac 1 2 (A(\nabla_X (f+g)-\j))\cdot (\nabla_X (f+g)-\j) \, \d X\d Y\d t\notag\\
&+\iint_{V_{Y,t}} {\langle f^*,(h(\cdot,Y,t)+f(\cdot,Y,t))\rangle -\langle (X\cdot\nabla_Y-\partial_t)h(\cdot,Y,t),f(\cdot,Y,t)\rangle}\, \d Y\d t\\
&{+\int_{U_X}\iint_{\partial V_{Y,t}} \bigl (\frac 1 2 f^2 + fh)(X,-1)\cdot N_{Y,t}\, \d\sigma_{Y,t}\d X}\biggr \},\notag
\end{align}
where the supremum still is with respect to $(f,\j,f^*)$ as in \eqref{e.rest}  subject to \eqref{f*norm} and \eqref{e.constraint.fj2}.

Consider  an arbitrary pair
$(b,\tilde f)\in C^\infty_{\K,0}(\overline{U_X\times V_{Y,t}})\times (W\cap C^\infty_{X,0}(\overline{U_X\times V_{Y,t}}))$. For each $\delta > 0$, denote
\begin{equation*}  
V_{Y,t}^\delta := \Ll\{ (Y,t) \in V_{Y,t}  \mid \dist(Y,t,\partial V_{Y,t}) \ge \delta \Rr\}.
\end{equation*}
For every $\delta > 0$ sufficiently small so that $V_{Y,t}^{2\delta} \neq \emptyset$, let $\psi_\delta \in C^\infty_0(V_{Y,t})$ denote a smooth function such that
\begin{align*}  
& \psi_\delta  \equiv 1 \quad \text{ on } V_{Y,t}^{2\delta}, \\
& \psi_\delta  \equiv 0 \quad \text{ on } (\mathbb R^m\times \mathbb R) \setminus V_{Y,t}^\delta. 
\end{align*}
We think of $\psi_\delta$ as a function on $U_X\times V_{Y,t} $ that does not depend on $X$. Given $(b,\tilde f)$ as above we introduce
\begin{equation}
\label{e.test.fr}
f_\delta=f_\delta(\tilde f,b):= \tilde f \psi_\delta+ (1-\psi_\delta)b.
\end{equation}
Then $ f_\delta \in C_{\K,0}^\infty(U_X\times V_{Y,t})$ and $(f_\delta-b)=0$ on $(U_X\times \partial V_{Y,t})$. As
\begin{equation*}  
\nabla_X  f_\delta = \psi_\delta \nabla_X \tilde f+(1-\psi_\delta)\nabla_X b,
\end{equation*}
we deduce that
\begin{equation*}  
||\tilde f- f_\delta||_{L_{Y,t}^2(V_{Y,t},H_X^1(U_X))}=||(\tilde f+b)(1-\psi_\delta)||_{L_{Y,t}^2(V_{Y,t},H_X^1(U_X))}\to 0,\: \mbox{as $\delta\to 0$}.
\end{equation*}

Consider $(\tilde f,\j,f^*,b)$ in the set
\begin{align*}
(W\cap C^\infty_{X,0}(\overline{U_X\times V_{Y,t}}))\times (L_{Y,t}^2(V_{Y,t},L^2_X(U_X)))^m\times L_{Y,t}^2(V_{Y,t},H_X^{-1}(U_X))\times C^\infty_{\K,0}(\overline{U_X\times V_{Y,t}}).
\end{align*}
Given $(b,\tilde f)$, $\delta > 0$, let $f=f_\delta$ be as above. Then with $f=f_\delta$ we have
\begin{align*}
& \iiint_{U_X\times V_{Y,t}}- \frac 1 2 (A\nabla_X (\tilde f+g)-\j)\cdot (\nabla_X (\tilde f+g)-\j) \, \d X\d Y\d t\notag\\
&\quad +\iint_{V_{Y,t}} {\langle f^*,(h(\cdot,Y,t)+\tilde f(\cdot,Y,t))\rangle-\langle (X\cdot\nabla_Y-\partial_t)h(\cdot,Y,t),\tilde f(\cdot,Y,t)\rangle}\, \d Y\d t
\notag\\
&\quad {+\int_{U_X}\iint_{\partial V_{Y,t}} \bigl (\frac 1 2 b^2+bh)(X,-1)\cdot N_{Y,t}\, \d\sigma_{Y,t}\d X}\notag\\
&= \iiint_{U_X\times V_{Y,t}}- \frac 1 2 (A\nabla_X (f+g)-\j)\cdot (\nabla_X (f+g)-\j) \, \d X\d Y\d t\notag\\
&\quad +\iint_{V_{Y,t}} {\langle f^*,(h(\cdot,Y,t)+f(\cdot,Y,t))\rangle-\langle (X\cdot\nabla_Y-\partial_t)h(\cdot,Y,t),f(\cdot,Y,t)\rangle}\, \d Y\d t
\notag\\
&\quad {+\int_{U_X}\iint_{\partial V_{Y,t}} \bigl (\frac 1 2 f^2+fh)(X,-1)\cdot N_{Y,t}\, \d\sigma_{Y,t}\d X} +R(\delta,\tilde f,b,g,f^*,h,f),
\end{align*}
where $R(\delta,\tilde f,b,g,f^*,h,f)$ denotes an error term which we can estimate. Indeed, let $M$ denote the largest of the four norms
\begin{equation*}
    \begin{split}
        &||g||_{L_{Y,t}^2(V_{Y,t},H_X^{1}(U_X))},\\  &||f^*||_{L_{Y,t}^2(V_{Y,t},H_X^{-1}(U_X))},\\
        &||(X\cdot\nabla_Y-\partial_t)h||_{L_{Y,t}^2(V_{Y,t},H_X^{-1}(U_X))},\\ &||f^* + g^*  - (X\cdot\nabla_Y-\partial_t)g||_{L_{Y,t}^2(V_{Y,t},H_X^{-1}(U_X))}.
    \end{split}
\end{equation*}
We also let
$$\Gamma(\tilde f,b):=||\tilde f||_{L_{Y,t}^2(V_{Y,t},H_X^1(U_X))}+||b||_{L_{Y,t}^2(V_{Y,t},H_X^1(U_X))}.$$
Then by the explicit construction of $f=f_\delta$,
\begin{align}
|R(\delta,\tilde f,b,g,f^*,h,f)|\leq \bigl(M+\Gamma(\tilde f,b)\bigr)||(\tilde f+b)(1-\psi_\delta)||_{L_{Y,t}^2(V_{Y,t},H_X^1(U_X))}.
\end{align}
 Based on this we can conclude that
\begin{align*}
& \iiint_{U_X\times V_{Y,t}}- \frac 1 2 (A\nabla_X (\tilde f+g)-\j)\cdot (\nabla_X (\tilde f+g)-\j) \, \d X\d Y\d t\notag\\
&+\iint_{V_{Y,t}} {\langle f^*,(h(\cdot,Y,t)+\tilde f(\cdot,Y,t))\rangle-\langle (X\cdot\nabla_Y-\partial_t)h(\cdot,Y,t),\tilde f(\cdot,Y,t)\rangle}\, \d Y\d t
\notag\\
&{+\int_{U_X}\iint_{\partial V_{Y,t}} \bigl (\frac 1 2 b^2+bh)(X,-1)\cdot N_{Y,t}\, \d\sigma_{Y,t}\d X} \notag\\
&\leq G^*(h)+\bigl(M+\Gamma(\tilde f,b)\bigr)||(\tilde f+b)(1-\psi_\delta)||_{L_{Y,t}^2(V_{Y,t},H_X^1(U_X))}.
\end{align*}
Let \begin{align*}
\tilde G^*(h) := \sup_{(\tilde f,\j,f^*,b)} &\biggl\{ \iiint_{U_X\times V_{Y,t}}- \frac 1 2 (A\nabla_X (\tilde f+g)-\j)\cdot (\nabla_X (\tilde f+g)-\j) \, \d X\d Y\d t\notag\\
&+\iint_{V_{Y,t}} {\langle f^*,(h(\cdot,Y,t)+\tilde f(\cdot,Y,t))\rangle-\langle (X\cdot\nabla_Y-\partial_t)h(\cdot,Y,t),\tilde f(\cdot,Y,t)\rangle}\, \d Y\d t
\notag\\
&{+\int_{U_X}\iint_{\partial V_{Y,t}} \bigl (\frac 1 2 b^2+bh)(X,-1)\cdot N_{Y,t}\, \d\sigma_{Y,t}\d X} \biggr \},
\end{align*}
where the supremum is taken with respect to all $(\tilde f,\j,f^*,b)$ in the set
\begin{align*}
(W\cap C^\infty_{X,0}(\overline{U_X\times V_{Y,t}}))\times (L_{Y,t}^2(V_{Y,t},L^2_X(U_X)))^m\times L_{Y,t}^2(V_{Y,t},H_X^{-1}(U_X))\times C^\infty_{\K,0}(\overline{U_X\times V_{Y,t}}),
\end{align*}
which satisfies $$\Gamma(\tilde f,b)\leq \Gamma, $$
for some large but fixed $\Gamma\geq 1$. Then,
$$\tilde G^*(h)\leq G^*(h)+\bigl(M+\Gamma\bigr)\sup_{\{(\tilde f,b):\ \Gamma(\tilde f,b)\leq \Gamma\}}||(\tilde f+b)(1-\psi_\delta)||_{L_{Y,t}^2(V_{Y,t},H_X^1(U_X))}.$$
However, by the above considerations, and dominated convergence, we see that
$$\bigl(M+\Gamma\bigr)\sup_{\{(\tilde f,b):\ \Gamma(\tilde f,b)\leq \Gamma\}}||(\tilde f+b)(1-\psi_\delta)||_{L_{Y,t}^2(V_{Y,t},H_X^1(U_X))}\to 0,$$
as $\delta\to 0$. In particular, to prove the lemma it suffices to prove that $\tilde G^*(h)\geq 0$.

To start the proof of that $\tilde G^*(h)\geq 0$, recall that $h\in W\cap C^\infty_{X,0}(\overline{U_X\times V_{Y,t}})$. Furthermore, in the definition of
$\tilde G^*(h)$ we have the supremum with respect to $\tilde f\in W\cap C^\infty_{X,0}(\overline{U_X\times V_{Y,t}})$. Hence, we can produce a lower bound on
$\tilde G^*(h)$ by choosing {$\tilde f=-h$} when considering the supremum as long as, say, $\Gamma\geq 2||h||_{L_{Y,t}^2(V_{Y,t},H_X^1(U_X))}$. Hence, selecting {$\tilde f = -h$} we obtain
\begin{align*}
\tilde G^*(h) \geq \sup_{(\j,f^*,b)} &\biggl\{ \iiint_{U_X\times V_{Y,t}} -\frac 1 2 (A(\nabla_X (h+g)-\j))\cdot (\nabla_X (h+g)-\j) \, \d X\d Y\d t\notag\\
&{+\iint_{V_{Y,t}}\langle (X\cdot\nabla_Y-\partial_t)h(\cdot,Y,t),h(\cdot,Y,t)\rangle\, \d Y\d t}\notag\\
&+\int_{U_X}\iint_{\partial V_{Y,t}} { (\frac 1 2 b^2+bh)}(X,-1)\cdot N_{Y,t}\, \d\sigma_{Y,t}\d X\biggr \},
\end{align*}
where the supremum is over every $(\j,f^*,b)$ as above.  However, as $h\in W\cap C^\infty_{X,0}(\overline{U_X\times V_{Y,t}})$ we have
\begin{align*}
&{\iint_{V_{Y,t}}\langle (X\cdot\nabla_Y-\partial_t)h(\cdot,Y,t),h(\cdot,Y,t)\rangle\, \d Y\d t}\notag\\
&+\int_{U_X}\iint_{\partial V_{Y,t}} { (\frac 1 2 b^2+bh)}(X,-1)\cdot N_{Y,t}\, \d\sigma_{Y,t}\d X\notag\\
&=\int_{U_X}\iint_{\partial V_{Y,t}} {\frac 12(b+h)^2}(X,-1)\cdot N_{Y,t}\, \d\sigma_{Y,t}\d X,
\end{align*}
and therefore
\begin{align*}
\tilde G^*(h) \geq \sup_{(\j,f^*,b)} &\biggl\{ \iiint_{U_X\times V_{Y,t}} -\frac 1 2 (A(\nabla_X (h+g)-\j))\cdot (\nabla_X (h+g)-\j) \, \d X\d Y\d t\notag\\
&+\int_{U_X}\iint_{\partial V_{Y,t}} \frac 12{(b+h)^2}(X,-1)\cdot N_{Y,t}\, \d\sigma_{Y,t}\d X\biggr \},
\end{align*}
where still the supremum is over  $(\j,f^*,b)$ as above. Using Lemma
\ref{lemma1c} below we have that
\begin{align}
\label{e.G*2abch}
\sup_{b} \int_{U_X}\iint_{\partial V_{Y,t}} \frac 12{(b+h)^2}(X,-1)\cdot N_{Y,t}\, \d\sigma_{Y,t}\d X\geq 0,
\end{align}
where the supremum is taken with respect to all
$$b\in W\cap C^\infty_{\K,0}(\overline{U_X\times V_{Y,t}}),\ ||b||_{L_{Y,t}^2(V_{Y,t},H_X^1(U_X))}\leq\Gamma.$$
Using this we can conclude that
\begin{equation*}  
\tilde G^*(h) \ge \sup \biggl\{ \iiint_{U_X\times V_{Y,t}} -\frac 1 2 (A(\nabla_X (h+g)-\j))\cdot (\nabla_X (h+g)-\j) \, \d X\d Y\d t\biggr \},
\end{equation*}
with the supremum ranging over all $\j \in (L_{Y,t}^2(V_{Y,t},L^2_X(U_X)))^m$ and  $f^* \in L_{Y,t}^2(V_{Y,t},H_X^{-1}(U_X))$  satisfying the constraint \eqref{e.constraint.fj2}. We now simply select $\j = \nabla_X (h+g) \in (L_{Y,t}^2(V_{Y,t},L^2_X(U_X)))^m$ and then $$f^* = \nabla_X\cdot(A \j) -g^* + (X\cdot\nabla_Y-\partial_t)g \in  L_{Y,t}^2(V_{Y,t},H_X^{-1}(U_X))$$ to conclude that $\tilde G^*(h) \ge 0$.
\end{proof}

\begin{lemma}\label{lemma1c} Assume that $h\in W(U_X\times V_{Y,t})\cap C^\infty_{X,0}(\overline{U_X\times V_{Y,t}})$. Then
\begin{align}
\sup_{b\in W\cap C^\infty_{\K,0} (U_X\times V_{Y,t})} \iiint_{U_X\times\partial V_{Y,t}} {(b+h)^2}(X,-1)\cdot N_{Y,t}\, \d X \d\sigma_{Y,t}\geq 0.
\end{align}
\end{lemma}
\begin{proof}
Let $\psi(s)\in C^\infty(\R)$ be such that $0\leq \psi \leq 1$,
\begin{equation*}
        \psi \equiv 1\ \mbox{on }[ 0,1],\ \psi \equiv 0\ \mbox{on }[ 2,\infty),
\end{equation*}
$|\psi'|\leq 2$ and such that $\sqrt{1-\psi^2}\in C^\infty(\R)$. Based on $\psi$ we introduce for $r$, $0\leq r <\infty$
\[
\psi_r(X,Y,t) := \psi\biggl( r\, \frac{\big((X,-1)\cdot N_{Y,t}\big)^-}{1+|X|^2} \biggr),
\]
where we use the notation $(s)^-:=\max\lbrace 0,-s \rbrace$ for $s\in\R$. As $h$ is smooth, and $U_X$ and $V_{Y,t}$ are bounded domains, we have
\begin{equation}\label{bdryfinite}
    \iiint_{U_X\times \partial V_{Y,t}} h^2|(X,-1)\cdot N_{Y,t}|\d X\d \sigma_{Y,t} < \infty.
\end{equation}
Let, for any $r\geq 0$,
\begin{equation}\label{br}
    {b_r := (\psi_r-1)h.}
\end{equation}
We claim that
\begin{equation}\label{bdryfinitea}
b_r\in W(U_X\times V_{Y,t}).
\end{equation}
To prove this it is enough to prove that
\begin{equation}\label{brinW}
    \|h\psi_r\|_{W(U_X\times V_{Y,t})} \leq c(1+r)\|h\|_{W(U_X\times V_{Y,t})},
\end{equation}
where $c=c(m,\psi)<\infty$. To see this we first note that
\begin{equation}
    |h\psi_r| \leq 2|h|,\   |\nabla_X h\psi_r| \leq |\nabla_X h| + cr|h|.
\end{equation}
Hence,
\begin{equation}
    \|h\psi_r\|_{L^2_{Y,t}(V_{Y,t},H^1_X(U_X))} \leq c(1+r) \|h\|_{L^2_{Y,t}(V_{Y,t},H^1_X(U_X))}.
\end{equation}
Using
\begin{equation}
    (X,-1)\cdot\nabla_{Y,t}(h\psi_r) = h(X,-1)\cdot\nabla_{Y,t}\psi_r + \psi_r (X,-1)\cdot\nabla_{Y,t}h,
\end{equation}
together with the estimates
\begin{equation}
    \|h(X,-1)\cdot\nabla_{Y,t}\psi_r\|_{L^2_{Y,t}(V_{Y,t},L^2_X(U_X))} \leq c\|h\|_{L^2_{Y,t}(V_{Y,t},H^1_X(U_X))},
\end{equation}
and
\begin{align*}
        &\|\psi_r (X,-1)\cdot\nabla_{Y,t}h\|_{L^2_{Y,t}(V_{Y,t},H^{-1}_X(U_X))}\\
         &\leq c \|\psi_r\|_{L^2_{Y,t}(V_{Y,t},H^1_X(U_X))} \|(X,-1)\cdot\nabla_{Y,t}h\|_{L^2_{Y,t}(V_{Y,t},H^{-1}_X(U_X))},
\end{align*}
we deduce \eqref{bdryfinitea}.

By construction, $b_r$ vanishes on $\partial_\K(U_X\times V_{Y,t})$. Together with \eqref{bdryfinitea}, this yields that $b_r\in W\cap C^\infty_{\K,0}(\overline{U_X\times V_{Y,t}})$. Furthermore,
\begin{equation*}
    \begin{split}
        \iiint_{U_X\times\partial V_{Y,t}} {(b_r+h)^2}(X,-1)\cdot N_{Y,t}\, \d X \d\sigma_{Y,t} &= \iiint_{U_X\times\partial V_{Y,t}} \psi_r^2h^2(X,-1)\cdot N_{Y,t}\, \d X \d\sigma_{Y,t}.
    \end{split}
\end{equation*}
Noting that the right hand side above has non-negative limit,
\begin{align*}
&\lim_{r\rightarrow\infty} \iiint_{U_X\times\partial V_{Y,t}} \psi_r^2h^2(X,-1)\cdot N_{Y,t}\, \d X \d\sigma_{Y,t}\\
&=\iiint_{U_X\times\partial V_{Y,t}} h^2\big((X,-1)\cdot N_{Y,t}\big)^-\, \d X \d\sigma_{Y,t}\geq 0, \end{align*}
we can complete the proof of the lemma.
\end{proof}

\subsection{Proof of Theorem \ref{weakdp1}: the quantitative estimate} Having proved that there exists a weak solution $u\in W(U_X\times V_{Y,t})$ to the problem in \eqref{dpweak} in the sense of Definition \ref{defwe}, we here derive the quantitative estimate stated in Theorem \ref{weakdp1}. Indeed, suppose that  $u\in W(U_X\times V_{Y,t})$ is a weak solution to the problem in \eqref{dpweak} in the sense discussed. Then first, simply using the weak formulation with any test function $\phi\in L_{Y,t}^2(V_{Y,t},{H}_{X,0}^{1}(U_X))$, such that $$||\phi||_{L_{Y,t}^2(V_{Y,t},{H}_{X,0}^{1}(U_X))}=1,$$ we deduce
    \begin{eqnarray*}
    {||(-X\cdot\nabla_Y+\partial_t)u||_{L_{Y,t}^2(V_{Y,t},{H}_X^{-1}(U_X))}\leq c\bigl (||\nabla_X u||_{L_{Y,t}^2(V_{Y,t},L^2(U_X))}+||g^*||_{L_{Y,t}^2(V_{Y,t},{H}_X^{-1}(U_X))}\bigr ).}
    \end{eqnarray*}
    Furthermore, using that $(u-g)$ is a valid test function we have
    \begin{align*}
     0 =&\iiint_{U_X\times V_{Y,t}}\ A(X,Y,t)\nabla_Xu\cdot \nabla_X(u-g)\, \d X \d Y \d t\notag\\
    &+\iint_{V_{Y,t}}\ \langle g^*(\cdot,Y,t)+(-X\cdot\nabla_Y+\partial_t)u(\cdot,Y,t), (u-g)(\cdot,Y,t)\rangle\, \d Y \d t.
\end{align*}
In particular,
 \begin{equation}\label{weak3llaha}
 \begin{split}
     0 =&\iiint_{U_X\times V_{Y,t}}\ A(X,Y,t)\nabla_X(u-g)\cdot \nabla_X(u-g)\, \d X \d Y \d t\\
    &+\iint_{V_{Y,t}}\ \langle (-X\cdot\nabla_Y+\partial_t)(u-g)(\cdot,Y,t), (u-g)(\cdot,Y,t)\rangle\, \d Y \d t\\
    &+\iiint_{U_X\times V_{Y,t}}\ A(X,Y,t)\nabla_Xg\cdot \nabla_X(u-g)\, \d X \d Y \d t\\
    &+\iint_{V_{Y,t}}\ \langle g^*(\cdot,Y,t)+(-X\cdot\nabla_Y+\partial_t)g(\cdot,Y,t), (u-g)(\cdot,Y,t)\rangle\, \d Y \d t.
\end{split}
\end{equation}
Again arguing as in \eqref{apa} and \eqref{apa+}, see \eqref{ex3}, we can deduce that
\begin{align}\label{weak3llahaha}
   \iint_{V_{Y,t}}\ \langle (-X\cdot\nabla_Y+\partial_t)(u-g)(\cdot,Y,t), (u-g)(\cdot,Y,t)\rangle\, \d Y \d t\geq 0.
\end{align}
Hence, using \eqref{weak3llaha}, \eqref{weak3llahaha}, \eqref{eq:Aellip}, { Cauchy-Schwarz and Young's inequality} we see that
 \begin{align*}
 ||\nabla_X(u-g)||_{L_{Y,t}^2(V_{Y,t},L^2(U_X))}&\leq c\bigl (||g||_{W(U_X\times V_{Y,t})}+||g^*||_{L_{Y,t}^2(V_{Y,t},{H}_X^{-1}(U_X))}\bigr )\notag\\
 &\quad +\epsilon ||(u-g)||_{L_{Y,t}^2(V_{Y,t},H_X^1(U_X))},
\end{align*}
where $\epsilon\in (0,1)$ is a degree of freedom and $c=c(m,\kappa,\epsilon)$ is a constant. Using Lemma \ref{lemma1a-} we can therefore conclude that
\begin{align}\label{weak3lla++}
 ||(u-g)||_{L_{Y,t}^2(V_{Y,t},H_X^1(U_X))}&\leq c\bigl (||g||_{W(U_X\times V_{Y,t})}+||g^*||_{L_{Y,t}^2(V_{Y,t},{H}_X^{-1}(U_X))}\bigr ).
\end{align}
Put together this proves the quantitative estimate in Theorem \ref{weakdp1}.

    \section{Proof of Theorem \ref{weakdp1++}, Theorem \ref{weakdp1gagain} and Theorem \ref{DP}}\label{Sec3}

    In this section we prove Theorem \ref{weakdp1++}, Theorem \ref{weakdp1gagain} and Theorem \ref{DP}. Throughout the section we let $U_X\subset\mathbb R^{m}$, $U_{Y}\subset \mathbb R^{m}$,  and $J\subset\mathbb R$ be bounded domains. In addition we assume that $U_X\subset\mathbb R^{m}$ is a bounded Lipschitz domain. We let the space $W(U_X\times U_Y\times J)$ be defined exactly as $W(U_X\times V_{Y,t})$ but with $V_{Y,t}$ replaced by $U_Y\times J$. The definition of weak solution to \eqref{dpweak} in the sense of Definition \ref{defwe}, but with $V_{Y,t}$ replaced by  $U_Y\times J$, is defined analogously. Let $\Omega\subset\mathbb R^m$, $m\geq 1$, be a (unbounded) Lipschitz domain as defined in \eqref{Lip}, {with boundary defined by a Lipschitz function $\psi$ with Lipschitz constant $M$}.

     We say that
    $u\in W_{\mathrm{loc}}(\Omega\times \mathbb R^{m}\times \mathbb R)$ if $u\in W(U_X\times U_{Y}\times J)$ whenever $U_X\subset\mathbb R^{m}$, $U_Y\subset\mathbb R^{m}$ are bounded domains, $J\subset\mathbb R$  is an open and bounded interval, and $U_X\times U_{Y}\times J\subset\Omega\times \mathbb R^{m}\times \mathbb R$.

     We let $W(\mathbb R^{N+1})=W(\mathbb R^m\times \mathbb R^{m}\times \mathbb R)$ be the closure of $C_0^\infty(\mathbb R^m\times \mathbb R^{m}\times \mathbb R)$ in the norm
\begin{equation}\label{space}
\begin{split}
 ||g||_{W(\mathbb R^{N+1})}&:= \biggl (\iint_{\mathbb R^m\times\mathbb R}(||g(\cdot,Y,t)||_{L^2(\mathbb R^{m})}^2+||\nabla_Xg(\cdot,Y,t)||_{L^2(\mathbb R^{m})}^2)\,\d Y\d t\biggr )^{1/2}\\
 &\quad + \biggl (\iint_{\mathbb R^{m}\times \mathbb R}||(-X\cdot\nabla_Y+\partial_t)g(\cdot,Y,t)||_{{H}_X^{-1}(\mathbb R^{m})}^2\, \d Y\d t\biggl )^{1/2}.
 \end{split}
 \end{equation}
 Given $\Omega\subset\mathbb R^m$ as above we let
 $W_{X,0}(\Omega\times \mathbb R^{m}\times \mathbb R)$ be the closure of all functions in $C_0^\infty(\overline{\Omega}\times \mathbb R^{m}\times \mathbb R)$, which are zero on
  $\partial\Omega\times \mathbb R^{m}\times \mathbb R$, in the norm of $W(\mathbb R^{N+1})=W(\mathbb R^m\times \mathbb R^{m}\times \mathbb R)$.

  \begin{definition}\label{deffa} Assume that $g\in W(\mathbb R^{m}\times \mathbb R^{m}\times \mathbb R)$, $g^*\in L^2_{Y,t}(\mathbb R^{m}\times \mathbb R,{H}_X^{-1}(\mathbb R^{m}))$. We say that $u$ is a weak solution to
    \begin{equation}\label{dpweak+g}
\begin{cases}
	\L u = g^*  &\text{in} \ \Omega\times \mathbb R^{m}\times \mathbb R, \\
      u = g  & \text{on} \ \partial\Omega\times \mathbb R^{m}\times \mathbb R,
\end{cases}
\end{equation}
if $u\in W_{\mathrm{loc}}(\Omega\times \mathbb R^{m}\times \mathbb R)$ and the following holds. First, \begin{eqnarray}\label{weak1g}
{\varphi}(u-g)\in W_0(\Omega\times \mathbb R^{m}\times \mathbb R),
    \end{eqnarray}
    whenever ${\varphi}\in C_{0}^\infty(\mathbb R^{m}\times \mathbb R^{m}\times \mathbb R)$. Second,
    \begin{equation}\label{weak3g}
    \begin{split}
     0 =&\iiint_{U_X\times U_Y\times J}\ A(X,Y,t)\nabla_Xu\cdot \nabla_X\phi\, \d X \d Y \d t\\
    &+\iint_{U_Y\times J}\ \langle g^*(\cdot,Y,t)+(-X\cdot\nabla_Y+\partial_t)u(\cdot,Y,t),\phi(\cdot,Y,t)\rangle\, \d Y \d t.
\end{split}
\end{equation}
for all $\phi\in L_{Y,t}^2(U_Y\times J,H_{X,0}^1(U_X))$, whenever $U_X\subset\mathbb R^{m}$, $U_Y\subset\mathbb R^{m}$ and $J\subset\mathbb R$ are bounded domains.
\end{definition}

    \subsection{Proof of  Theorem \ref{weakdp1++}} Recall that the natural family of dilations for the operator $\L$, $(\delta_r)_{r>0}$, on $\R^{N+1}$,
is
\begin{equation}\label{dil.alpha.i}
 \delta_r (X,Y,t) =(r X, r^3 Y,r^2 t),
\end{equation}
for $(X,Y,t) \in \R^{N +1}$,  $r>0$.   In the following we will, without loss of generality, assume that $0\in \partial\Omega$.  We let, for $R>0$,
\begin{equation}\label{def.Omega.Delta.fr}
\begin{split}
    U_{X}^R&:=\Omega\cap\{X=(x,x_m)  \mid  |x_i|<R,\  \psi(x)<x_m<4MR\},\\
    \Delta_{X}^R&:=\partial\Omega\cap\{X=(x,x_m)  \mid  |x_i|<R,\  -4MR<x_m<4MR\}.
\end{split}
\end{equation}
Furthermore, we let $V_{Y,t}\subset\mathbb R^{m}\times \mathbb R$ be a bounded domain with a boundary  which is $C^{1,1}$-smooth and contains $(0,0)\in \mathbb R^{m}\times \mathbb R$. For $R>0$ we introduce
\begin{equation}\label{def.Omega.Delta.fr++}
    V_{Y,t}^R:=\Omega\cap\{(Y,t)  \mid  (R^{-3}Y,R^{-2}t)\in V_{Y,t}\}.
\end{equation}
Then, for each $R>0$, $U_X^R\subset\mathbb R^{m}$ is a bounded Lipschitz domain and
$V_{Y,t}^R\subset\mathbb R^{m}\times \mathbb R$ is a bounded domain with a boundary  which is $C^{1,1}$-smooth. Furthermore, $U_X^{\tilde R}\times V_{Y,t}^{\tilde R}\subseteq U_X^{R}\times V_{Y,t}^{R}$ whenever $\tilde R\leq  R$ and
$$\Omega\times \mathbb R^{m}\times \mathbb R=\bigcup_{R>0} U_X^R\times V_{Y,t}^R, $$
i.e., the domains $\{U_X^R\times V_{Y,t}^R\}_R$ exhaust $\Omega\times \mathbb R^{m}\times \mathbb R$. Finally, we let
$W_{\Delta_X^R,0}(U_X^R\times V_{Y,t}^R)$ be the closure in $L_{Y,t}^2(V_{Y,t}^R,H_X^1(U_X^R))$ of all functions in  $C^\infty(\overline{U_X^R\times V_{Y,t}^R})$ which vanish on
$\Delta_X^R\times \overline{V_{Y,t}^R}$.

 To start the proof, we let, given $R>0$, $\phi_R$, $0\leq\phi_R \leq 1$,  be a smooth function such that $\phi_R=1$ on $\Delta_X^{R/2}\times \overline{V_{Y,t}^R}$ and
    $\phi_R=0$ on $\partial_\K (U_X^R\times V_{Y,t}^R)\setminus (\Delta_X^{3R/4}\times \overline{V_{Y,t}^R})$. Assuming $g\in W(\mathbb R^{m}\times \mathbb R^{m}\times \mathbb R)$, $g^*\in L^2_{Y,t}(\mathbb R^{m}\times \mathbb R,{H}_X^{-1}(\mathbb R^{m}))$, we see that if we let $g_R=g\phi_R$ and $g_R^*=g^*\phi_R$, then $g_R\in W(U_X^R\times V_{Y,t}^R)$, $g_R^*\in L^2_{Y,t}(V_{Y,t}^R,{H}_X^{-1}(U_X^R))$ and $g_R\in W_{Y,t,0}(U_X^R\times V_{Y,t}^R)$. Using Theorem \ref{weakdp1} we see that there exists a unique weak solution $u_R$ to
    \begin{equation}\label{dpweak+l}
\begin{cases}
	\L u_R = g_R^*  &\text{in} \ U_X^R\times V_{Y,t}^R, \\
      u_R = g_R  & \text{on} \ \partial_{\K}(U_X^R\times V_{Y,t}^R),
\end{cases}
\end{equation}
    in the sense of Definition \ref{defwe}, i.e.
                      \begin{eqnarray}\label{weak1l}
u_R\in W(U_X^R\times V_{Y,t}^R),\ (u_R-g_R)\in  W_0(U_X^R\times V_{Y,t}^R),
    \end{eqnarray}
    and
\begin{equation}\label{weak3l}
    \begin{split}
     0 =&\iiint_{U_X^R\times V_{Y,t}^R}\ A \nabla_Xu_R\cdot \nabla_X\phi\, \d X \d Y \d t\\
    &+\iint_{V_{Y,t}^R}\ \langle g_R^*(\cdot,Y,t)+(-X\cdot\nabla_Y+\partial_t)u_R(\cdot,Y,t),\phi(\cdot,Y,t)\rangle_R\, \d Y \d t,
\end{split}
\end{equation}
for all $ \phi\in L_{Y,t}^2(V_{Y,t}^R,H_{X,0}^1(U_X^R))$ and where now $\langle \cdot,\cdot\rangle_R=\langle \cdot,\cdot\rangle_{H_X^{-1}(U_X^R),H_{X,0}^{1}(U_X^R)}$ is the duality pairing in $H_X^{-1}(U_X^R)$. Note that \eqref{weak1l} implies that  \begin{eqnarray}\label{weak1l+}
(u_R-g_R)\in  W_{\Delta_X^R,0}(U_X^R\times V_{Y,t}^R).
    \end{eqnarray}
    Furthermore, using \eqref{weak3l} we see, by the estimate stated in Theorem \ref{weakdp1}, that there exists a constant $c$, independent of $u_R$ and $g_R$ but depending on $m$, $\kappa$, and $U_X^R\times V_{Y,t}^R$ such that
    \begin{align*}
    ||u_R||_{W(U_X^R\times V_{Y,t}^R)}&\leq c\bigl (||g_R||_{W(U_X^R\times V_{Y,t}^R)}+||g_R^*||_{L_{Y,t}^2(V_{Y,t}^R,H_{X}^{-1}(U_X^R))}\bigr )\\
    &\leq c\bigl (||g||_{W(\mathbb R^m\times \mathbb R^{m}\times \mathbb R)}+||g^*||_{L^2_{Y,t}(\mathbb R^{m}\times \mathbb R,{H}_X^{-1}(\mathbb R^{m}))}\bigr ).
    \end{align*}
    Fix $\tilde R>0$ and consider $R\geq \tilde R$. Then the inequality in the last display implies that
    \begin{eqnarray*}
    ||u_R||_{W(U_X^{\tilde R}\times V_{Y,t}^{\tilde R})}\leq c(\tilde R)\bigl (||g||_{W(\mathbb R^m\times \mathbb R^{m}\times \mathbb R)}+||g^*||_{L^2_{Y,t}(\mathbb R^{m}\times \mathbb R,{H}_X^{-1}(\mathbb R^{m}))}\bigr ).
    \end{eqnarray*}
    In particular, if $R\geq \tilde R$ then $\{u_R\}$ is uniformly bounded in the norm of $W(U_X^{\tilde R}\times V_{Y,t}^{\tilde R})$. Using this we see that there
    exists a subsequence of $\{u_R\}$, which we denote by $\{u_{\tilde R}^j\}$, such that
    \begin{eqnarray*}
    u_{\tilde R}^j\to \tilde u_{\tilde R}\mbox{ weakly in }W(U_X^{\tilde R}\times V_{Y,t}^{\tilde R}),
    \end{eqnarray*}
     and $\tilde u_{\tilde R}$ satisfies $\tilde u_{\tilde R}\in W(U_X^{\tilde R}\times V_{Y,t}^{\tilde R})$ and is such that $(\tilde u_{\tilde R}-g)\in  W_{\Delta_X^{\tilde R},0}(U_X^{\tilde R}\times V_{Y,t}^{\tilde R})$. Furthermore, $\tilde u_{\tilde R}$ satisfies \eqref{weak3l} with $U_X^{R}\times V_{Y,t}^{R}$ replaced by $U_X^{\tilde R}\times V_{Y,t}^{\tilde R}$ and for all $ \phi\in L_{Y,t}^2(V_{Y,t}^{\tilde R},H_{X,0}^1(U_X^{\tilde R}))$. In particular, $\tilde u_{\tilde R}$ is a weak solution in
    $U_X^{\tilde R}\times V_{Y,t}^{\tilde R}$ with the correct boundary data, i.e. $g$, on $\Delta_X^{\tilde R}$. Next, consider a sequence $\{\tilde R_i\}$, $\tilde R_i\to \infty$ as $i\to\infty$. Then by the argument outlined,
    \begin{eqnarray*}
    u_{\tilde R_i}^j\to \tilde u_{\tilde R_i}\mbox{ weakly in }W(U_X^{\tilde R_i}\times V_{Y,t}^{\tilde R_i}),
    \end{eqnarray*}
    and $\tilde u_{\tilde R_i}$ is a solution in $U_X^{\tilde R_i}\times V_{Y,t}^{\tilde R_i}$ in the sense discussed. Hence, taking a diagonal subsequence, denoted
    $\{u_{\tilde R_l^l}\}$,  we see that $u_{\tilde R_l^l}\to \tilde u$ as $l\to \infty$, and $\tilde u$ is such that, for every $R>0$
     \begin{eqnarray}\label{weak1gaa}
\tilde u\in W(U_X^R\times V_{Y,t}^R),\ (\tilde u-g)\in  W_{\Delta_X^R,0}(U_X^R\times V_{Y,t}^R),
    \end{eqnarray}
    and
\begin{equation}\label{weak3gaa}
    \begin{split}
     0 =&\iiint_{U_X^R\times V_{Y,t}^R}\ A \nabla_X\tilde u\cdot \nabla_X\phi\, \d X \d Y \d t\\
    &+\iint_{V_{Y,t}^R}\ \langle g^*(\cdot,Y,t)+(-X\cdot\nabla_Y+\partial_t)\tilde u(\cdot,Y,t),\phi(\cdot,Y,t)\rangle_R\, \d Y \d t,
\end{split}
\end{equation}
for all $\phi\in L_{Y,t}^2(V_{Y,t}^R,H_{X,0}^1(U_X^R))$. In particular, $\tilde u$ is a weak solution to the problem to the problem in \eqref{dpweak+g} in the sense of Definition \ref{deffa}. This completes the proof of the
existence of $\tilde u$ and hence the proof of the theorem.

    \subsection{Proof of Theorem \ref{weakdp1gagain}}
    Assume that $u\in W_{\mathrm{loc}}(\Omega\times \mathbb R^{m}\times \mathbb R)\cap L^\infty(\Omega\times \mathbb R^{m}\times \mathbb R)$ is a weak solution to the problem in \eqref{dpweak+g}, with $g\equiv 0\equiv g^*$, in the sense of Definition \ref{deffa}. We want to prove that $u\equiv 0$ on $\Omega\times \mathbb R^{m}\times \mathbb R$. We can without loss of generality assume that $0\in\partial\Omega$. Let
    $V_{Y,t}\subset\mathbb R^{m}\times \mathbb R$ be a bounded domain with a boundary  which is $C^{1,1}$-smooth and contains $(0,0)\in \mathbb R^{m}\times \mathbb R$. Given $R>0$ we let $U_{X}^R$ and $V_{Y,t}^R$ be defined as in \eqref{def.Omega.Delta.fr} and \eqref{def.Omega.Delta.fr++}, respectively.

    We first note that the weak maximum principle holds in bounded domains $U_{X}^{R}\times V_{Y,t}^{R}$. Indeed, by Theorem \ref{weakdp1} we have uniqueness to the problem in \eqref{dpweak} in $U_{X}^{R}\times V_{Y,t}^{R}$. Therefore it suffices to prove the  weak maximum principle for the regularized operators
\begin{eqnarray}\label{e-kolm-nd+ddd}
   \L_{\epsilon}:=\nabla_X\cdot(A^\epsilon(X,Y,t)\nabla_X)+X\cdot\nabla_{Y}-\partial_t.
    \end{eqnarray}
    Here $\epsilon>0$ is small and $A^\epsilon$ is a regularization of $A$, constructed by convolution of $A$ with respect to an approximation of the identity with parameter $\epsilon$, such that $A^\epsilon\to A$ a.e. on compact subsets of $\mathbb R^{N+1}$ as $\epsilon\to 0$. The weak maximum principle for $\L_{\epsilon}$ in $U_{X}^{R}\times V_{Y,t}^{R}$ can be proved as in Lemma 6.1 in \cite{LN1}.

    Let $R>0$ be  large enough to ensure that
    $A=I_m$ on $(\Omega\times \mathbb R^{m}\times \mathbb R)\setminus(U_{X}^{R}\times V_{Y,t}^{R})$. Then $\L=\K$ on $(\Omega\times \mathbb R^{m}\times \mathbb R)\setminus(U_{X}^{R}\times V_{Y,t}^{R})$. Let
    $${M^*}:=\sup_{{U_{X}^{R}\times V_{Y,t}^{R}}} |u|.$$
    Let $\eta>0$  be given and assume that $u\geq 2\eta$ at some point in $\Omega\times \mathbb R^{m}\times \mathbb R$. Note that the maximum principle that we are to prove is known to hold for $\K$. Therefore, using H{\"o}lder continuity estimates for $\K$, as well as estimates for the fundamental solution of $\K$, see for example Theorem 3.2 and Lemma 4.17 in \cite{LN1}, we can conclude that there exists $0<\delta=\delta(\eta,{M^*})\ll 1$
such that $ u<\eta$ in $(\Omega\times \mathbb R^{m}\times \mathbb R)\setminus (U_{X}^{R/\delta}\times V_{Y,t}^{R/\delta})$.  Hence $u\geq 2\eta$ at some point in $(U_{X}^{R/\delta}\times V_{Y,t}^{R/\delta})$. However, applying  the maximum principle in the bounded domain $(U_{X}^{R/\delta}\times V_{Y,t}^{R/\delta})$ we see
$u\leq \eta$ in $(U_{X}^{R/\delta}\times V_{Y,t}^{R/\delta})$. This yields a contradiction. Hence $u \leq 0$ on $\Omega\times \mathbb R^{m}\times \mathbb R$. Arguing analogously we can prove that $u \geq 0$ on $\Omega\times \mathbb R^{m}\times \mathbb R$, and hence $u\equiv 0$ on $\Omega\times \mathbb R^{m}\times \mathbb R$. The proof of Theorem \ref{weakdp1gagain} is complete.

    \subsection{Proof of Theorem \ref{DP}} Assume that $A$ satisfies \eqref{eq:Aellip} with constant $\kappa$ as well as \eqref{comp}. Let  $\varphi\in W(\mathbb R^{N+1})\cap C_0(\mathbb R^{N+1})$. Consider the regularized operator $\L_{\epsilon}$ introduced in \eqref{e-kolm-nd+ddd}. Applying Theorem \ref{weakdp1++}, Theorem \ref{weakdp1gagain} as well as Theorem 3.1 in \cite{LN1}, in
    $\Omega\times \mathbb R^{m}\times \mathbb R$ and for $\L_{\epsilon}$, we can conclude that there exists a unique  bounded weak solution $u_\epsilon\in W_{\mathrm{loc}}(\Omega\times \mathbb R^{m}\times \mathbb R)\cap C(\overline{\Omega}\times \mathbb R^m\times\mathbb R)$, in the sense of Definition \ref{deffa}, to the problem
    \begin{equation}\label{dpweakeps}
\begin{cases}
	\mathcal{L}_{\epsilon} u_{\epsilon} =0  &\text{in} \ \Omega\times \mathbb R^{m}\times \mathbb R, \\
u_{\epsilon}=\varphi & \text{on} \ \partial \Omega\times \mathbb R^{m}\times \mathbb R.
\end{cases}
\end{equation}
Furthermore, there exists, for every $(X,Y,t)\in \Omega\times \mathbb R^m\times \mathbb R $, a unique probability
measure  $\omega_{\K}^\epsilon(X,Y,t,\cdot)$ on $\partial\Omega\times \mathbb R^m\times \mathbb R $ such that
\begin{eqnarray*}
u_\epsilon(X,Y,t)=\iiint_{\partial\Omega\times \mathbb R^m\times \mathbb R }\varphi(\tilde X,\tilde Y,\tilde t)\d \omega_{\K}^\epsilon(X,Y,t,\tilde X,\tilde Y,\tilde t).
\end{eqnarray*}
It is readily seen that $u_\epsilon\to u$ in $W_{\mathrm{loc}}(\Omega\times \mathbb R^{m}\times \mathbb R)$, that
$\omega_{\K}^\epsilon\to \omega_{\K}$ weakly on $\partial\Omega\times \mathbb R^m\times \mathbb R$ in the sense of measures, that $u\in W_{\mathrm{loc}}(\Omega\times \mathbb R^{m}\times \mathbb R)\cap L^\infty(\Omega\times \mathbb R^{m}\times \mathbb R)$ is the unique  bounded weak solution, in the sense of Definition \ref{deffa}, to the problem in \eqref{dpweakeps} and that
\begin{eqnarray*}
u(X,Y,t)=\iiint_{\partial\Omega\times \mathbb R^m\times \mathbb R }\varphi(\tilde X,\tilde Y,\tilde t)\d \omega_{\K}(X,Y,t,\tilde X,\tilde Y,\tilde t).
\end{eqnarray*}
Therefore, to prove the theorem it only remains to prove that $u\in C(\overline{\Omega}\times \mathbb R^m\times\mathbb R)$, i.e. to prove that $u$ is continuous up to the boundary. To prove this we first introduce some notation.

Recall that the class of operators $\L$ is closed under the  group law
\begin{equation}\label{e70}
(\tilde X,\tilde Y,\tilde t)\circ (X, Y,t):=(\tilde X+X,\tilde Y+Y-t\tilde X,\tilde t+t),\ (X,Y,t),\ (\tilde X,\tilde Y,\tilde t)\in \R^{N+1}.
\end{equation}
Given $(X,Y,t)\in \R^{N+1}$ we let \begin{equation}\label{kolnormint}
\|(X,Y, t)\|:=|(X,Y)|\!+|t|^{\frac{1}{2}},\ |(X,Y)|:=\big|X\big|+\big|Y\big|^{1/3}.
\end{equation}
We recall the following pseudo-triangular
inequalities: there exists a positive constant ${c}$ such that
\begin{eqnarray}\label{e-ps.tr.in}
\quad \|(X,Y,t)^{-1}\|\le {c}  \| (X,Y,t) \|,\ \|(X,Y,t)\circ (\tilde  X,\tilde  Y,\tilde  t)\| \le  {c}  (\| (X,Y,t) \| + \| (\tilde  X,\tilde  Y,\tilde  t)
\|),
\end{eqnarray}
whenever $(X,Y,t),(\tilde  X,\tilde  Y,\tilde  t)\in \R^{N+1}$. Using \eqref{e-ps.tr.in} it  follows immediately that
\begin{equation} \label{e-triangularap}
    \|(\tilde  X,\tilde  Y,\tilde  t)^{-1}\circ (X,Y,t)\|\le c \, \|(X,Y,t)^{-1}\circ (\tilde  X,\tilde  Y,\tilde  t)\|,
\end{equation}
whenever $(X,Y,t),(\tilde  X,\tilde  Y,\tilde  t)\in \R^{N+1}$. Let
\begin{equation}\label{e-ps.distint}
    d((X,Y,t),(\tilde X,\tilde Y,\tilde t)):=\frac 1 2\bigl( \|(\tilde X,\tilde Y,\tilde t)^{-1}\circ (X,Y,t)\|+\|(X,Y,t)^{-1}\circ (\tilde X,\tilde Y,\tilde t)\|).
\end{equation}
Using \eqref{e-triangularap} it follows that
\begin{equation}\label{e-ps.dist}
   \|(\tilde X,\tilde Y,\tilde t)^{-1}\circ (X,Y,t)\|\approx d((X,Y,t),(\tilde X,\tilde Y,\tilde t))\approx \|(X,Y,t)^{-1}\circ (\tilde X,\tilde Y,\tilde t)\|,
\end{equation}
with constants of comparison independent of $(X,Y,t),(\tilde X,\tilde Y,\tilde t)\in \R^{N+1}$. Again using \eqref{e-ps.tr.in} we also see that
\begin{equation} \label{e-triangular}
    d((X,Y,t),(\tilde X,\tilde Y,\tilde t))\le {c} \bigl(d((X,Y,t),(\hat X, \hat Y,\hat t))+d((\hat X, \hat Y,\hat
t),(\tilde X,\tilde Y,\tilde t))\bigr ),
\end{equation}
whenever $(X,Y,t),(\hat X, \hat Y,\hat t),(\tilde X,\tilde Y,\tilde t)\in \R^{N+1}$, and hence $d$ is a symmetric quasi-distance.

Let $Q:=(-1,1)^m\times (-1,1)^m\times (-1,1)$ and let
$$Q_r=\delta_r Q:=\{(r X, r^3 Y,r^2 t)\mid (X,Y,t)\in Q\}.$$
Given a point $(X_0,Y_0,t_0)\in \R^{N+1}$ we let
$$ Q_r(X_0,Y_0,t_0):=(X_0,Y_0,t_0)\circ Q_r:=\{(X_0,Y_0,t_0)\circ(X,Y,t)\mid (X,Y,t)\in Q_r\}.$$

We are now ready to prove that $u\in C(\overline{\Omega}\times \mathbb R^m\times\mathbb R)$. Let $(\tilde X_0,\tilde Y_0,\tilde t_0)\in \partial\Omega\times \mathbb R^m\times \mathbb R$. Given
$\delta>0$ we first choose $\tilde r=\tilde r(\delta)$ so that
$$|\varphi(X,Y,t)-\varphi(X_0,Y_0,t_0)|<\delta/2,$$
whenever $(X,Y,t)\in (\overline{\Omega}\times \mathbb R^m\times \mathbb R)\cap  Q_{\tilde r}(X_0,Y_0,t_0)$. Let $\phi$ be a test function satisfying $0\leq\phi\leq 1$, with support in $ Q_{\tilde r}(X_0,Y_0,t_0)$, such that $\phi\equiv 1$ on
$ Q_{\tilde r/2}(X_0,Y_0,t_0)$. Let $\hat \varphi(X,Y,t):=\phi(X,Y,t) \bigl(\varphi(X,Y,t)-\varphi(X_0,Y_0,t_0)\bigr)$ and $\tilde \varphi(X,Y,t):=\bigl(1-\phi(X,Y,t)\bigr)\bigl(\varphi(X,Y,t)-\varphi(X_0,Y_0,t_0)\bigr)$. Now let
$\hat u_\epsilon$ and $\tilde u_\epsilon $ be the unique bounded weak solutions to the problems in \eqref{dpweakeps} but with $\varphi$ replaced with $\hat \varphi$ and $\tilde \varphi$, respectively. It follows by uniqueness  that $u_{\epsilon}(X,Y,t)-\varphi(X_0,Y_0,t_0)=\hat u_\epsilon(X,Y,t)+\tilde u_\epsilon(X,Y,t)$ whenever $(X,Y,t)$ is in the closure of $\Omega\times \mathbb R^m\times \mathbb R$. Furthermore,
\begin{equation}\label{eq:dir_norm}
 	||\hat u_\epsilon||_{L^\infty(\Omega\times \mathbb R^m\times \mathbb R)}\leq\delta/4,
 \end{equation}
 by the weak maximum principle for $\L_\epsilon$. Using Theorem 3.2 in \cite{LN1}
\begin{align*}
|\tilde u_\epsilon(X,Y,t)| \le c \left(\frac{d((X,Y,t),(X_0,Y_0,t_0))}{\tilde r}\right)^{\alpha},
\end{align*}
whenever $(X,Y,t)\in (\overline{\Omega}\times \mathbb R^m\times \mathbb R)\cap  Q_{\tilde r/2}(X_0,Y_0,t_0)$ and with constants $c$ and $\alpha$ independent of $\epsilon$. We now choose
$r=r(\delta)$ so that $c \left(r/\tilde r\right)^{\alpha}<\delta/4$. Then
$$\sup_{\epsilon>0}|u_{\epsilon}(X,Y,t)-\varphi(X,Y,t)|<\delta,$$
whenever $(X,Y,t)\in (\overline{\Omega}\times \mathbb R^m\times \mathbb R)\cap  Q_{r}(X_0,Y_0,t_0)$. In particular, the family $\{u_\eps\}$ is uniformly continuous on compact sets up to the boundary and with modulus of  uniform continuity independent of $\eps$. Thus, $u_\eps$ converges uniformly to $u$ in any neighborhood of the boundary and, in particular
$u\in C(\overline{\Omega}\times \mathbb R^m\times\mathbb R)$.  This completes the proof of Theorem \ref{DP}.

\end{document}